\documentclass[11pt,a4paper]{amsart}

\usepackage[pdftitle={Critical weak-$L^{p}$ differentiability of singular integrals}, pdfauthor={Luigi Ambrosio, Augusto C. Ponce and Rémy Rodiac}]{hyperref}

\usepackage{amsfonts}

\usepackage[T1]{fontenc}

\usepackage{enumerate}

\usepackage[abbrev,backrefs]{amsrefs}

\usepackage[nameinlink]{cleveref}
\usepackage{constants}

\usepackage[utf8]{inputenc}

\usepackage{graphicx}
\usepackage{amssymb}
\usepackage{esint}
\usepackage{mathtools}

\usepackage{palatino}

\usepackage{color}
\usepackage{todonotes}

\def \M {\mathcal{M}}
\newcommand{\abs}[1]{\mathopen\lvert#1\mathclose\rvert}
\newcommand{\bigabs}[1]{\bigl\lvert#1\bigr\rvert}
\newcommand{\Bigabs}[1]{\Bigl\lvert#1\Bigr\rvert}
\newcommand{\biggabs}[1]{\biggl\lvert#1\biggr\rvert}
\newcommand{\norm}[1]{\mathopen\lVert#1\mathclose\rVert}
\newcommand{\bignorm}[1]{\mathopen\big\lVert#1\mathclose\big\rVert}
\newcommand{\Bignorm}[1]{\mathopen\Big\lVert#1\mathclose\Big\rVert}

\newcommand{\N}{{\mathbb N}}
\newcommand{\R}{{\mathbb R}}
\newcommand{\Z}{{\mathbb Z}}

\newcommand{\cH}{\mathcal{H}}
\newcommand{\cM}{\mathcal{M}}

\DeclareMathOperator{\curl}{curl}
\DeclareMathOperator{\supp}{supp}

\DeclareMathOperator{\diam}{diam}

\DeclareMathOperator{\Div}{div}
\DeclareMathOperator{\Curl}{curl}

\DeclareMathOperator{\Trace}{tr}
\DeclareMathOperator{\Int}{int}
\DeclareMathOperator{\Dom}{dom}
\newcommand{\Lebesgue}[1]{\widetilde{#1}}
\newcommand{\apD}[1]{{D}_{\mathrm{ap}}{#1}}
\newcommand{\apDD}[1]{{D}_{\mathrm{ap}}^{2}{#1}}
\DeclareMathOperator*{\aplim}{ap\,lim}
\newcommand{\seminorm}[1]{\mathopen[ #1 \mathclose]}
\newcommand{\bigseminorm}[1]{\bigl[ #1 \bigr]}
\newcommand{\dif}{\,\mathrm{d}}

\newcommand{\la}{_\mathrm{a}}
\newcommand{\ls}{_\mathrm{s}}
\newcommand{\lw}{_\mathrm{w}}
\newcommand{\loc}{_\mathrm{loc}}
\newcommand{\e}{\mathrm{e}}

\newcommand{\meas}[1]{\left| #1 \right|}

\theoremstyle{plain}
\newtheorem{proposition}{Proposition}[section]
\newtheorem{lemma}[proposition]{Lemma}
\newtheorem{theorem}[proposition]{Theorem}
\newtheorem{corollary}[proposition]{Corollary}

\theoremstyle{definition}

\newtheorem{remark}[proposition]{Remark}
\newtheorem{definition}{Definition}[section]

\theoremstyle{remark}
\newtheorem{example}{Example}[section]

\newtheorem*{Claim}{Claim}

\numberwithin{equation}{section}

\title[Critical weak-$L^{p}$ differentiability of singular integrals]{Critical weak-$L^{p}$ differentiability\\of singular integrals}

\author{Luigi Ambrosio}
\address{
Luigi Ambrosio\hfill\break\indent
 Scuola Normale Superiore\hfill\break\indent
 Piazza Cavalieri 7\hfill\break\indent 
 56126 Pisa, Italy}
\email{luigi.ambrosio@sns.it}

\author{Augusto C. Ponce}
\address{
Augusto C. Ponce\hfill\break\indent
 Université catholique de Louvain\hfill\break\indent
 Institut de Recherche en Mathématique et Physique\hfill\break\indent
 Chemin du cyclotron 2, L7.01.02\hfill\break\indent
1348 Louvain-la-Neuve, Belgium}
\email{Augusto.Ponce@uclouvain.be}

\author{Rémy Rodiac}
\address{
Rémy Rodiac\hfill\break\indent
 Université catholique de Louvain\hfill\break\indent
 Institut de Recherche en Mathématique et Physique\hfill\break\indent
 Chemin du cyclotron 2, L7.01.02\hfill\break\indent
1348 Louvain-la-Neuve, Belgium}
\email{Remy.Rodiac@uclouvain.be}

\begin{document}

\begin{abstract}
We establish that for every function $u \in L^1\loc(\Omega)$ whose distributional Laplacian $\Delta u$ is a signed Borel measure in an open set \(\Omega\) in \(\R^{N}\), the distributional gradient \(\nabla u\) is differentiable almost everywhere in \(\Omega\) with respect to the weak-\(L^{\frac{N}{N-1}}\) Marcinkiewicz norm.
We show in addition that the absolutely continuous part of $\Delta u$ with respect to the Lebesgue measure equals zero almost everywhere on the level sets $\{ u= \alpha\}$ and  $\{\nabla u=e \}$, for every $\alpha \in \R$ and $e \in \R^N$. 
Our proofs rely on an adaptation of Calderón and Zygmund's singular-integral estimates inspired by subsequent work by Haj\l asz.
\end{abstract}



\maketitle

\section{Introduction and main results}

Let $\Omega$ be an open set in $\R^N$ with $N\geq 2$.{}
This paper was originally motivated by the following question of H.~Brezis's:
Given \(u \in L^{1}\loc(\Omega)\) whose distributional Laplacian satisfies \(\Delta u \in L^{1}\loc(\Omega)\), is it true that for any $\alpha \in \R$ one has
\begin{equation}
\label{eqQuestion}
\Delta u = 0 
\quad \text{almost everywhere on \(\{ u=\alpha\}\)~?}
\end{equation}
The answer is straightforward when \(\Delta u \in L^{p}\loc(\Omega)\) for some \(1 < p < \infty\), since in this case \(u\) belongs to the Sobolev space \(W^{2, p}\loc(\Omega)\).{}
One then has \(D^{2} u= 0\) almost everywhere on the level set \(\{\nabla u = 0\}\) and the latter contains \(\{u = \alpha\}\) except for a negligible set; see Theorem~4.4 in \cite{EvansGariepy}.
As \(\Delta u = \Trace{(D^{2} u)}\) is the trace of \(D^{2} u\), assertion~\eqref{eqQuestion} is satisfied.

When one merely has \(\Delta u \in L^{1}\loc(\Omega)\), it need not be true that \(u\) belongs to \(W^{2, 1}\loc(\Omega)\).{}
The question above has nevertheless a positive answer that includes its generalization when \(\Delta u\) is merely a measure.{}
By the latter, we mean that there exists a locally finite Borel measure \(\lambda\) in \(\Omega\), possibly signed, such that
\[{}
\int_{\Omega}{u \Delta\varphi}
= \int_{\Omega}{\varphi \dif\lambda}
\quad \text{for every \(\varphi \in C_{c}^{\infty}(\Omega)\)}
\]
and \emph{we always identify \(\Delta u\) with \(\lambda\)}.
One shows in this case that the distributional gradient \(\nabla u\) belongs to \(L^{1}\loc(\Omega; \R^{N})\) and has an approximate derivative  at almost every point \(y \in \Omega\), denoted
\[{}
\apDD{u}(y) \vcentcolon= \apD{(\nabla u)}(y).
\] 
From the definition of the approximate derivative which we recall in \Cref{2} below, \(\apDD{u}(y)\) is a linear transformation from \(\R^{N}\) to \(\R^{N}\).
The fact that \(u\) belongs to the Sobolev space \(W\loc^{1, 1}(\Omega)\) follows from standard elliptic regularity theory, see Theorem~5.1 and Proposition~6.11 in \cite{Ponce2016}, while the existence of the approximate derivative of \(\nabla u\) is a consequence of the Remark on p.~129 by Calderón and Zygmund~\cite{CalderonZygmund}.

The answer to H.~Brezis's question can be seen as a consequence of an identification of \(\Delta u\) in terms of \(\apDD{u}\):

\begin{theorem}\label{th:main1}
If $u \in L^{1}\loc(\Omega)$ is such that $\Delta u$ is a locally finite Borel measure in $\Omega$, 
then the approximate derivative \(\apDD{u}\) satisfies
\begin{equation}
	\label{eqmain1}
	(\Delta u)\la{}
= \Trace{(\apDD{u})} \dif x.
\end{equation}
Hence, for every $\alpha \in \R$ and \(e \in \R^{N}\),
\begin{equation}
	\label{eqmain2}
(\Delta u)\la = 0
\quad \text{almost everywhere on \(\{ u = \alpha\} \cup\{\nabla u = e\}\).}
\end{equation}
\end{theorem}

Here, \((\Delta u)\la\) is the absolutely continuous part of \(\Delta u\) with respect to the Lebesgue measure \(\mathrm{d} x\). 
In particular, when \(\Delta u \in L^{1}\loc(\Omega)\), one has 
\[{}
\Delta u = \Trace{(\apDD{u})} 
\quad \text{almost everywhere in \(\Omega\)}
\]
and \eqref{eqQuestion} holds.

Assertion~\eqref{eqmain2} follows from \eqref{eqmain1} and a standard property of the approximate derivative on level sets; see \eqref{eqDensityPoint_ter}.
Marano and Mosconi~\cite{MaranoMosconi2015} have established \eqref{eqmain2} using alternatively the $L^1$-differentiability of \(\nabla u\) by Alberti, Bianchini and Crippa~\cite{AlbertiBianchiniCrippa2014a} in the sense of \eqref{eqLpDifferentiability} below.
Identity~\eqref{eqmain1} has been proved independently by Raita~\cite{Raita2017}.
His proof also relies on~\cite{AlbertiBianchiniCrippa2014a} and includes a generalization to elliptic operators of any integer order; see also~\cite{GmeinederRaita}.

One might wonder whether \Cref{th:main1} is a consequence of a stronger locality property of the divergence of vector fields, namely
the absolutely continuous part of the divergence of a vector field being \(0\) almost everywhere on level sets of the
vector field itself.
A simple application of Alberti's Lusin-type theorem \cite{Alberti} shows that such a property is not true for general vector fields, despite the fact that it holds for distributional gradients of Sobolev functions:

\begin{example}\label{rem:alberti}
Let $\Omega$ be a bounded open set in $\R^2$ and let us consider the vector field $V(x_1,x_2)=(x_2, -x_1)$.
For every $\phi\in C^1_c(\Omega)$, the continuous vector field
\[
W = (V - \nabla\phi)^\perp,
\]
where \((a_{1}, a_{2})^{\perp} \vcentcolon= (-a_{2}, a_{1})\), satisfies
\[{}
\Div{W}
= 2\dif x
\quad \text{in the sense of distributions in \(\Omega\)}. 
\]
This is a consequence of Schwarz's theorem which implies that \(\Div{(\nabla\phi)^\perp} = 0\) in the sense of distributions in \(\Omega\) for any \(\phi\).
By Alberti's theorem~\cite{Alberti}, for any $0 < \epsilon < \abs{\Omega}$ we can find some $\phi\in C^1_c(\Omega)$ such that the Lebesgue measure of \(\{\nabla\phi\neq V\}\) is less than \(\epsilon\).{}
With such a choice of \(\phi\), the set $\{V-\nabla\phi=0\}$ where the vector field \(W\) equals \(0\) has positive Lebesgue measure.
\end{example}

Our strategy to prove identity~\eqref{eqmain1} has been inspired by the work~\cite{Hajlasz1996} of Haj\l asz's that  has some analogy with the pointwise estimates by Calderón and Zygmund~\cite{CalderonZygmund1961} and De~Vore and Sharpley~\cite{DeVoreSharpley}; see also \cite{KuusiMingione}.
It provides some reinvigorating insight concerning existence of the approximate derivative in connection with the theory of singular integrals.
To this end, we first observe that for a smooth homogeneous kernel \(K\) of order \(-(N-1)\) in \(\R^{N} \setminus \{0\}\), that is
\[
K(x)
= \frac{1}{\abs{x}^{N-1}} K\Bigl( \frac{x}{\abs{x}} \Bigr) 
\quad \text{for every \(x \in \R^{N} \setminus \{0\}\),}
\]
the convolution \(K * \mu\) is defined almost everywhere in \(\R^{N}\) for every finite Borel measure \(\mu\) in \(\R^{N}\).{}
More precisely, the complement of the set
\[{}
\Dom{(K*\mu)}
\vcentcolon= \biggl\{ x \in \R^{N} : \int_{\R^{N}}\abs{K(x - y)} \dif\abs{\mu}(y) < \infty  \biggr\}
\]
is negligible with respect to the Lebesgue measure.

Note that \(K * \mu\) belongs to \(L^{1}\loc(\R^{N})\), but need not have a distributional gradient in \(L^{1}\loc (\R^{N}; \R^{N})\). 
Haj\l asz made the observation that a singular-integral estimate of \((\nabla K) * \mu\) can be formulated in terms of a Lipschitz-type estimate of \(K * \mu\) with variable coefficient.
Existence almost everywhere of the approximate derivative \(\apD{(K * \mu)}\) can then be straightforwardly obtained using Rademacher's theorem.

This approach applies more generally to kernels \(K\) that satisfy
\begin{equation}
	\label{eqKernel}
\abs{K(x)}
\le \frac{A}{\abs{x}^{N-1}}
\quad \text{and} \quad
\abs{D^{2} K(x)}
\le \frac{B}{\abs{x}^{N+1}}
\quad \text{for every \(x \in \R^{N} \setminus \{0\}\),}
\end{equation}
where \(A, B > 0\) are constants.
We prove a quantitative version of Calderón-Zygmund's approximate-differentiability property, which can be written as

\begin{theorem}\label{theoremADSingularIntegralAlternative}
	Let \(K \in C^{2}(\R^{N} \setminus \{0\})\) be any function that satisfies \eqref{eqKernel}.{}
	If \(\mu\) is a finite Borel measure in $\R^N$, then there exists a measurable function \(I : \R^{N} \to [0, \infty]\) in the Marcinkiewicz space \(L^{1}\lw(\R^{N})\) of weak-\(L^{1}\) functions such that
	\[
	\abs{K * \mu(x) - K * \mu(y)}
	\le (I(x) + I(y)) \abs{x - y}
	\quad \text{for every \(x, y \in \Dom{(K * \mu)}\)}
	\]
	and
	\[
	\seminorm{I}_{L^{1}\lw(\R^{N})}
	\le C \norm{\mu}_{\cM(\R^{N})},
	\]
	for some constant \(C > 0\) depending on \(A\), \(B\) and \(N\).
	Hence, \(K * \mu\) is approximately differentiable almost everywhere in \(\R^{N}\) and
	\[
	\abs{\apD{(K * \mu)}}
	\le 2 I
	\quad \text{almost everywhere in \(\R^{N}\).}
	\]
\end{theorem}

We recall that the total-variation norm of a finite Borel measure \(\mu\)  in an open set \(\Omega\) is
\[{}
\norm{\mu}_{\cM(\Omega)}
\vcentcolon= \abs{\mu}(\Omega)
= \int_{\Omega}{\dif\abs{\mu}}
\]
and, for every \(1 \le p < \infty\), the weak-\(L^{p}\) Marcinkiewicz quasinorm of a measurable function \(f\) in \(\Omega\) is  
\[{}
\seminorm{f}_{L^{p}\lw(\Omega)}
\vcentcolon= \sup{\bigl\{\, t \, \bigabs{\{\abs{f} > t\}}^{\frac{1}{p}} : t>0 \bigr\}},
\]
where \(\abs{E}\) denotes the Lebesgue measure of \(E\).

The choice of \(I\) is far from being unique or canonical.
Compared to~\cite[Lemma 9]{Hajlasz1996}, our function \(I\) satisfies a true weak-\(L^{1}\) estimate in the entire space \(\R^{N}\) that comes from a uniformization principle in \Cref{4'} below.{}
Such a global property was not stated nor needed in \cite{Hajlasz1996}, whose focus was on the existence of the approximate derivative of \(K * \mu\).{}
In our case, the identification of \(\Trace{(\apDD{u})}\) as the absolutely continuous part of \(\Delta u\) relies on an approximation argument based on the weak-\(L^{1}\) estimate of \(\apDD{u}\).{}

Under the additional assumption that \(\mu\) belongs to \(L^{p}(\R^{N})\) for some \(1 < p < \infty\), by standard singular-integral estimate of \((\nabla K) * \mu\) the distributional gradient \(\nabla(K * \mu)\) belongs to \(L^{p}(\R^{N}; \R^{N})\). 
In such a case, \(K * \mu\) satisfies a Lipschitz-type estimate as above with an explicit coefficient \(I\) in \(L^{p}(\R^{N})\) that involves the maximal function \(\cM\abs{\nabla(K * \mu)}\), see \eqref{eqChoiceILp}, with
\[{}
\norm{I}_{L^{p}(\R^{N})}
\le C' \norm{\mu}_{L^{p}(\R^{N})}.
\]

Although \Cref{theoremADSingularIntegralAlternative} is enough for proving \Cref{th:main1}, there is a notion which is stronger than approximate differentiability and is adapted to \(L^{p}\loc\) functions, namely \(L^{p}\) differentiability.
The goal is to determine whether at points \(y \in \R^{N}\) one has
\begin{equation}
	\label{eqLpDifferentiability}
\lim_{r \to 0}{\, \frac{1}{r} \, \frac{\norm{v - T_{y}^{1}v}_{L^{p}(B_{r}(y))}}{\norm{1}_{L^{p}(B_{r}(y))}}}
= \lim_{r \to 0}{\, \frac{1}{r} \biggl(\fint_{B_{r}(y)}{\abs{v - T_{y}^{1}v}^{p}} \biggr)^{\frac{1}{p}}}
=  0,
\end{equation}
where \(\fint_{B_{r}(y)}\) denotes the average integral over the ball and  \(T_{y}^{1}v\) is the first-order Taylor approximation of \(v\) at \(y\) defined by
\[{}
T_{y}^{1}v(x)
\vcentcolon= v(y) + \apD{v}(y)[x - y]
\quad \text{for every \(x \in \R^{N}\),}
\]
provided that the approximate derivative \(\apD{v}(y)\) exists.

A fundamental result by Calderón and Zygmund~\cite[Theorem~12]{CalderonZygmund1961}, see also \cite[Theorem~6.2]{EvansGariepy}, asserts that every \(v \in W^{1, 1}\loc(\R^{N})\) is \(L^{\frac{N}{N-1}}\)-differentiable at almost every point \(y \in \R^{N}\) and 
\begin{equation}
	\label{eqIdentificationApproximateDerivative}
\apD{v}(y)[h] = \nabla v(y) \cdot h
\quad \text{for every \(h \in \R^{N}\).}{}
\end{equation}
Under the assumptions of \Cref{theoremADSingularIntegralAlternative}, it may happen that \(K * \mu\) does not belong to \( W^{1, 1}\loc(\R^{N})\) and just misses the imbedding in \(L^{\frac{N}{N-1}}\loc(\R^{N})\).
Such an example is given by \(K(z) = 1/\abs{z}^{N - 1}\) and \(\mu = \delta_{a}\) with \(a \in \R^{N}\).{}
Nonetheless, Alberti, Bianchini and Crippa prove in~\cite{AlbertiBianchiniCrippa2014a} that $K * \mu$ is always $L^p$-differentiable almost everywhere in $\R^{N}$ in the range $1 \leq p < \frac{N}{N-1}$.

Concerning the critical exponent \(p = \frac{N}{N-1}\), a variant of Young's inequality easily implies that \(K * \mu\) belongs to the Marcinkiewicz space \(L^{\frac{N}{N-1}}\lw(\R^{N})\), which is locally contained in all \(L^{p}\)~spaces for \(1 \le p < \frac{N}{N-1}\).{}
One may thus wonder whether there is some notion of differentiability that could handle such a critical imbedding.{}
The answer is affirmative and is our next

\begin{theorem}
	\label{theoremDifferentiability}
	If \(K \in C^{2}(\R^{N} \setminus \{0\})\) satisfies \eqref{eqKernel} and \(\mu\) is a finite Borel measure in \(\R^{N}\), then \(K * \mu\) is weak-\(L^{\frac{N}{N-1}}\) differentiable at almost every point \(y \in \R^{N}\) in the sense that
	\[{}
	\lim_{r \to 0}{\, \frac{1}{r} \, \frac{\seminorm{K * \mu - T_{y}^{1}(K * \mu)}_{L^{\frac{N}{N-1}}\lw(B_{r}(y))}}{\seminorm{1}_{L^{\frac{N}{N-1}}\lw(B_{r}(y))}}} = 0.
	\]
	\end{theorem}

The normalization factor in the denominator satisfies
\[{}
\seminorm{1}_{L^{\frac{N}{N-1}}\lw(B_{r}(y))} 
= |B_{r}(y)|^{\frac{N-1}{N}}
= d_{N} r^{N - 1},
\] 
for some constant \(d_{N} > 0\).{}
Our proof of this weak-\(L^{\frac{N}{N-1}}\) differentiability property of \(K * \mu\) relies on the Calderón-Zygmund decomposition of \(\mu\) and the \(L^{\frac{N}{N-1}}\)~differentiability of Sobolev functions.
After completing this paper we have been informed by J.~Verdera that the methods used in his joint work with Cufí~\cite{CufiVerdera}, where they investigate new fine properties of functions \(u\) such that \(\Delta u\) is a measure, can be adapted to yield an alternative proof of \Cref{theoremDifferentiability}; see their comments on p.\@~1087 in~\cite{CufiVerdera} and also \cite{Verdera}.{}
While there is no hope of having \(L^{\frac{N}{N-1}}\) differentiability in full generality,
under additional ellipticity assumptions on the differential operator associated to \(K\), Gmeineder and Raita prove in~\cite{GmeinederRaita} that \(K * \mu\) belongs to \(L^{\frac{N}{N-1}}(\R^{N})\) and is \(L^{\frac{N}{N-1}}\)-differentiable in the usual sense.

We aim at a self-contained presentation, reproducing for the benefit of the reader also some intermediate results
already present in the literature.
The paper is then organized as follows.
We explain in \Cref{2} the connection between approximate differentiability and the Lipschitz-type condition used by Haj\l asz in~\cite{Hajlasz1996}. 
In \Cref{3} we recall the singular estimates for \(K * \mu\) when \(\mu \in L^{2}(\R^{N})\), based on the Fourier transform. 
In \Cref{4,sectionProofProposition} we obtain a weak-$L^1$ estimate for the approximate derivative of \(K * \mu\) when \(\mu\) is a measure, following the approach of Calderón and Zygmund's. 
We prove \Cref{theoremADSingularIntegralAlternative,theoremDifferentiability} in \Cref{4',sectionDifferentiability}, respectively.

In \Cref{5}, we apply the singular-integral estimates to identify the second-order approximate derivative of the Newtonian potential and the solution of the Dirichlet problem with measure data. 
We then prove identity~\eqref{eqmain1} in \Cref{sectionProofThm}.
In \Cref{7} we give two applications of \Cref{th:main1}. 
The first one is a new proof of a property of level sets of subharmonic functions by Frank and Lieb~\cite{FrankLieb}. 
The second one concerns the description of limiting vorticities of the Ginzburg-Landau model with magnetic field in bounded open subsets \(\Omega\) of \(\R^{2}\) that extends previous result by Sandier and Serfaty~\cite{SandierSerfaty, SandierSerfaty2003} in the \(L^{p}\) setting.
Based on regularity results by Caffarelli and Salazar~\cite{CaffarelliSalazar:2002}, we deduce that a limiting vorticity \(\mu \in L^{1}\loc(\Omega)\) can be written as
\[{}
\mu = \sum_{j \in J}{m_{j} \chi_{U_{j}}}
\quad \text{almost everywhere in \(\Omega\)},
\]
for some disjoint family of open subsets \(U_{j} \subset \Omega\), where \(m_{j}\) is the constant value of the limiting induced magnetic field on \(U_{j}\).{}

\section{Approximate differentiability via Lipschitz-type estimates}\label{2}

We recall that the approximate limit \(c \vcentcolon=\aplim\limits_{x\to y}{v}\) of a measurable function  \(v : \R^{N} \to \R^{m}\) at \(y\) is defined by the property
\begin{equation*}
\lim_{r \to 0}{\frac{\meas{A_{\epsilon}(v, c) \cap B_{r}(y)}}{\meas{B_{r}(y)}}}
= 0
\quad \text{for every \(\epsilon > 0\),}
\end{equation*}
with
\[{}
A_{\epsilon}(v, c)
\vcentcolon= \bigl\{ x \in \R^{N} : |v(x)- c|>\epsilon \bigr\}.
\]
In the terminology of Measure theory, \(y\) must be a density point of the set \(\R^N\setminus A_{\epsilon}(v, c)\) for every \(\epsilon > 0\). 
By Lebesgue's density theorem, one has $\aplim\limits_{x\to y}{v} = v(y)$ almost everywhere in $\R^N$ and in particular, for all $\alpha \in\R^{m}$,
\begin{equation*}
\aplim_{x\to y}{v} = \alpha \quad\text{almost everywhere on \(\{v = \alpha\}\).}
\end{equation*}

Accordingly, the approximate derivative of \(v\) at \(y\) is a linear transformation \(\apD{v}(y) : \R^{N} \to \R^{m}\) such that
\[{}
\aplim_{x \to y}{\frac{v(x) - v(y) - \apD{v}(y)[x - y]}{\abs{x - y}}}
= 0.
\]
Notice that the existence of the approximate derivative at $y$ implies the existence of the approximate limit of \(v\) at \(y\), with
\[{}
\aplim_{x \to y}{v}
= v(y).
\]
Lebesgue's density theorem can be invoked again, see for instance 
\cite{EvansGariepy}*{Theorem~6.3}, to get for all $\alpha\in\R^{m}$,
\begin{equation}\label{eqDensityPoint_ter}
\apD{v}=0\quad\text{almost everywhere on \(\{v=\alpha\}\).}
\end{equation}
These concepts extend to functions defined on measurable subsets of \(\R^{N}\)
and both \(\aplim\limits_{x\to y}{v}\) and \(\apD{v}(y)\) are uniquely defined at each density  point of the domain.

Haj\l asz's strategy to prove approximate differentiability from a Lipschitz-type estimate with variable coefficient is based on the following

\begin{proposition}\label{appdiff}
	Let \(E \subset \R^{N}\) and \(v : E \to \R\) be a measurable function.
	If there exists a measurable function \(I : E \to {[}0, \infty{)}\) such that
	\[{}
	\abs{v(x) - v(y)}
	\le (I(x) + I(y)) \abs{x - y}
	\quad \text{for every \(x, y \in E\),}
	\]
	then \(v\) is approximately differentiable almost everywhere in \(E\) and its approximate derivative satisfies
	\[{}
	\abs{\apD{v}}
	\le 2 I
	\quad \text{almost everywhere in \(E\).}
	\]
\end{proposition}

\begin{proof}
	Given \(\epsilon > 0\) and \(\alpha > 0\), the set 
	\[{}
	M = \{\alpha \le I \le \alpha + \epsilon\}.
	\]
	is a measurable subset of $E$. 
	Since \(v\) is Lipschitz-continuous on \(M\) with Lipschitz constant \(2 (\alpha + \epsilon)\), there exists a Lipschitz-continuous function \(h : \R^{N} \to \R\) with the same Lipschitz constant and such that \(h = v\) on \(M\).{}
	Observe that if \(y \in M\) is a density point of \(M\) and \(h\) is differentiable at \(y\), then \(v\) has an approximate derivative at \(y\) and \(\apD{v}(y) = Dh(y)\).{}
	Thus,
	\[{}
	\abs{\apD{v}(y)}
	= \abs{Dh(y)}
	\le 2(\alpha + \epsilon){}
	\le 2(I(y) + \epsilon).
	\]
	Since  by Lebesgue's density theorem almost every point of \(M\) is a density point of \(M\) and by Rademacher's theorem \(h\) is differentiable almost everywhere in \(\R^{N}\),{}
	we deduce from the observation above that \(v\) is approximately differentiable at almost every point of \(M\) and
	\[{}
	\abs{\apD{v}}
	\le 2(I + \epsilon){}
	\quad \text{almost everywhere in \(M\).}
	\]
	Since \(E\) can be covered by a countable union of such sets \(M\), we thus have
	\[{}
	\abs{\apD{v}}
	\le 2(I + \epsilon){}
	\quad \text{almost everywhere in \(E\),}
	\]
	and the estimate follows since \(\epsilon > 0\) is arbitrary.
\end{proof}

The Lipschitz-type estimate of \Cref{appdiff} is satisfied for example by Sobolev functions:

\begin{proposition}\label{MaximalandSobolev}
If \(v : \R^{N} \to \R\) is a measurable function such that \(v \in W^{1, 1}\loc(\R^{N})\), then
\[{}
\abs{\Lebesgue{v}(x) - \Lebesgue{v}(y)}
\le 2^{N} \bigl(\cM{\abs{\nabla v}}(x) + \cM{\abs{\nabla v}}(y)\bigr) \abs{x - y}
\]
for every Lebesgue points \(x\) and \(y\) of \(v\), where \(\Lebesgue{v}\) denotes the precise representative of \(v\).
\end{proposition}

The Hardy-Littlewood maximal function $\cM{f} : \R^{N} \to [0, \infty]$ of a locally summable function $f$ is defined by
\begin{equation*}
\cM{f}(x)
\vcentcolon= \sup{\biggl\{\fint_{B_{r}(x)}{|f|} : r > 0 \biggr\}}.
\end{equation*}
We also recall that \(x \in \R^{N}\) is a Lebesgue point of \(f\) and \(\Lebesgue{f}(x)\) 
is the value of the precise representative of \(f\) at \(x\) whenever
\[{}
\lim_{r \to 0}{\fint_{B_{r}(x)}{\abs{f - \Lebesgue{f}(x)}}}
= 0.
\]
The classical Lebesgue differentiation theorem implies that almost every point in \(\R^{N}\) is a Lebesgue point of \(f\) and 
\begin{equation}
\label{eqLDT}
\Lebesgue{f} = f
\quad \text{almost everywhere in \(\R^{N}\).}
\end{equation}
Notice also that, as a simple consequence of the Markov-Chebyshev inequality, for all \(x \in \R^{N}\) one has the implication
$$
\lim_{r \to 0}{\fint_{B_{r}(x)}{\abs{f - \Lebesgue{f}(x)}}}=0
\quad\Longrightarrow\quad
\aplim_{y\to x}f(y)=\Lebesgue{f}(x),
$$
but the converse implication does not hold in general.

A natural problem is to find sufficient conditions that identify Lebesgue points of a given function.
For example, when \(v \in W\loc^{1, 1}(\R^{N})\) as above, it follows from Remark~3.82 of \cite{AFP} that every \(x \in \R^{N}\) such that 
\[{}
\cM\abs{\nabla v}(x) < \infty{}
\]
is a Lebesgue point of \(v\).{}
Hence, the inequality in \Cref{MaximalandSobolev} is valid for every \(x, y \in \{\cM\abs{\nabla v} < \infty\}\).
Observe however that it does not provide the identification~\eqref{eqIdentificationApproximateDerivative} between \(\apD{v}\) and \(\nabla v\), which goes back to the work by Calderón and Zygmund~\cite{CalderonZygmund1962} that includes the case of functions with bounded variation (BV); see also \cite{AFP}*{Theorem~3.83}.

We sketch the proof of \Cref{MaximalandSobolev} for the convenience of the reader; see also
\cite{AFP}*{Theorem~5.34} for a different argument.
A variant of \Cref{MaximalandSobolev} based on the sharp maximal function can be found for example in \cite{DeVoreSharpley}*{Theorem~2.5}.
We also refer the reader to \cite{KuusiMingione} for applications of pointwise inequalities of this type in the setting on nonlinear Potential theory.

\resetconstant
\begin{proof}[Proof of \Cref{MaximalandSobolev}]
	Assume temporarily that \(v\) is smooth.
	By the Fundamental theorem of Calculus and Fubini's theorem,
	\[{}
	\biggabs{v(x) - \fint_{B_{r}(\frac{x+y}{2})}{v}}
	\le \int_{0}^{1}{\fint_{B_{r}(\frac{x+y}{2})}{\abs{\nabla v(x + t(z - x))} \abs{x - z} \dif z} \dif t},
	\]
	for any \(r > 0\).
	Taking \(r = \abs{x - y}/2\), one has \(B_{r}(\frac{x+y}{2}) \subset B_{2r}(x)\).{}
	Moreover, for every \(z\in B_{r}(\frac{x+y}{2})\), \(\abs{x - z} \le \abs{x - y}\).
	Hence,
	\[{}
	\biggabs{v(x) - \fint_{B_{r}(\frac{x+y}{2})}{v}}
	\le \frac{1}{\abs{B_{1}} r^{N}} \int_{0}^{1}{\int_{B_{2r}(x)}{\abs{\nabla v(x + t(z - x))} \dif z} \dif t} \; \abs{x - y},
	\]
	where \(B_{1} \vcentcolon= B_{1}(0)\) is the unit ball centered at \(0\).{}
	Making the change of variables \(\xi = x + t(z - x)\) between \(z\) and \(\xi\) in the right-hand side and using the definition of \(\cM\abs{\nabla v}(x)\), one gets
	\[{}
	\begin{split}
	\biggabs{{v}(x) - \fint_{B_{r}(\frac{x+y}{2})}{v}}
	& \le \frac{1}{\abs{B_{1}} r^{N}} \int_{0}^{1}{\int_{B_{2tr}(x)}{\abs{\nabla v(\xi)} \frac{\dif \xi}{t^{N}}} \dif t} \; \abs{x - y}\\
	& \le 2^{N} \int_{0}^{1}{\cM\abs{\nabla v}(x)} \dif t \; \abs{x - y}
	= 2^{N} \cM\abs{\nabla v}(x) \, \abs{x - y}.
	\end{split}
	\] 
	A similar estimate holds for the point \(y\) and one concludes using the triangle inequality.
	
	For a measurable function \(v\) as in the statement, one can apply the conclusion to the smooth function \(\rho_{n} * v\), where 
	\((\rho_{n})_{n \in \N}\) is a sequence of mollifiers of the form \(\rho_{n}(z) \vcentcolon= \frac{1}{\epsilon_{n}^{N}} \rho({z}/{\epsilon_{n}})\) for some fixed nonnegative function \(\rho \in C_{c}^{\infty}(B_{1})\) such that \(\int_{B_{1}}{\rho} = 1\) and \((\epsilon_{n})_{n \in \N}\) is a sequence of positive numbers that converges to zero. 
	As \(n \to \infty\), one has
	\[{}
	\rho_{n} * v(x) \to \Lebesgue{v}(x)
	\quad\text{for every Lebesgue point \(x \in \R^{N}\).}
	\]
	Moreover, the pointwise estimate \(\abs{\nabla(\rho_{n} * v)} \le \rho_{n} * \abs{\nabla v}\) implies that
	\[{}
	\cM\abs{\nabla(\rho_{n} * v)}(x) 
	\le \cM(\rho_{n} * \abs{\nabla v})(x)
	\le \cM\abs{\nabla v}(x)
	\quad \text{for every \(x \in \R^{N}\).}
	\qedhere
	\]
\end{proof}

\section{Weak differentiability of $K\ast \mu$ in the $L^{2}$ case}\label{3}

By the first estimate in \eqref{eqKernel}, one can write \(K\) as a sum of \(L^{1}\) and \(L^{\infty}\) functions, for instance:
\begin{equation}
	\label{eqKDecomposition}
	K = K \chi_{B_{1}} + K \chi_{\R^{N} \setminus B_{1}}.{}
\end{equation}
In particular, the convolution \(K * \mu\) is defined almost everywhere in \(\R^{N}\){}
for every \(\mu \in L^{1}(\R^{N})\) and, more generally, for finite Borel measures, and belongs to \((L^{1} + L^{\infty})(\R^{N})\).{}
The fact that the distributional derivative $\nabla (K\ast \mu)$ belongs to $L^2(\R^{N}; \R^{N})$ when $\mu \in L^2(\R^N)$ is a known property in Harmonic analysis that can be proved using the Fourier transform. 

\begin{proposition}
	\label{L2caseNew}
	If \(K \in C^{2}(\R^{N} \setminus \{0\})\) satisfies \eqref{eqKernel} and \(\mu \in (L^{1} \cap L^{2})(\R^{N})\), then the distributional gradient of $K \ast \mu$ belongs to \(L^{2}(\R^{N}; \R^{N})\) and we have
	\[{}
	\norm{\nabla(K * \mu)}_{L^{2}(\R^{N})}
	\le C \norm{\mu}_{L^{2}(\R^{N})},
	\]
	for some constant \(C > 0\) depending on \(A\), \(B\) and \(N\).
\end{proposition}

We recall that the Fourier transform of a function \(f \in L^{1}(\R^{N})\) is defined for every \(\xi \in \R^{N}\) by
\[{}
\mathcal{F}f(\xi){}
= \int_{\R^{N}}{\e^{- 2\pi \imath x \cdot \xi} f(x) \dif x},
\]
and has a unique continuous extension to every function in \(L^{2}(\R^{N})\).
We give a proof of the proposition above for the sake of completeness.
It relies on the following property of the Fourier transform of \(K\); see Theorem~2.4 in Chapter~3 of \cite{SteinShakarchi4}:

\begin{lemma}
	\label{lemmaKHomogeneous}
	If \(K \in C^{2}(\R^{N} \setminus \{0\})\) satisfies \eqref{eqKernel}, then its Fourier transform \(\mathcal{F}K\) is continuous in \(\R^{N} \setminus \{0\}\) and there exists \(C' > 0\) such that
	\[{}
	 \abs{\mathcal{F}K(\xi)}
	 \le \frac{C'}{\abs{\xi}}
	 \quad \text{for every \(\xi \in \R^{N} \setminus \{0\}\).}
	\]
\end{lemma}

\resetconstant
\begin{proof}[Proof of \Cref{lemmaKHomogeneous}]
	To prove the continuity of \(\mathcal{F}K\) in \(\R^{N} \setminus \{0\}\), one relies on a smooth counterpart of the decomposition \eqref{eqKDecomposition}.
	For this purpose, take \(\psi \in C_{c}^{\infty}(B_{2})\) such that \(0 \le \psi \le 1\) in \(\R^{N}\) and \(\psi = 1\) in \(B_{1}\), and write
	\[{}
	K = K \psi_{r} + K (1-\psi_{r}),
	\]
	where \(\psi_{r}(x) = \psi(x/r)\) for \(r > 0\).
	Since \(K\psi_{r} \in L^{1}(\R^{N})\), the Fourier transform \(\mathcal{F}(K\psi_{r})\) is a bounded continuous function in \(\R^{N}\).{}
	
	Interpolation in \eqref{eqKernel} gives one the first-order counterpart
\begin{equation}
	\label{eqKernel1}
\abs{D K(x)}
\le \frac{\C}{\abs{x}^{N}}
\quad \text{for every \(x \in \R^{N} \setminus \{0\}\).}
\end{equation}
	Since the function \(K_{\infty, r} \vcentcolon= K (1-\psi_{r})\) vanishes in \(B_{r}\), we thus have 
	\[{}
	\abs{\Delta K_{\infty, r}} \le \frac{\Cl{cte-479}}{\abs{x}^{N+1}} \, \chi_{\R^{N} \setminus B_{r}}(x)
	\quad \text{for every \(x \in \R^{N}\),}
	\]
	for some constant \(\Cr{cte-479} > 0\) independent of \(r\).
	In particular, \(\Delta K_{\infty, r}\) belongs to \(L^{1}(\R^{N})\), hence its Fourier transform is also a bounded continuous function in \(\R^{N}\).
	From the identity
	\[{}
	\mathcal{F}(\Delta K_{\infty, r})
	= - 4 \pi^{2} \abs{\xi}^{2} \mathcal{F}(K_{\infty, r}),
	\]
	we deduce that \(\mathcal{F}(K_{\infty, r})\) is continuous in \(\R^{N} \setminus \{0\}\), hence so is \(\mathcal{F}(K)\).{}
	
	To obtain the pointwise estimate of \(\mathcal{F}(K)\), observe that for every \(\xi \ne 0\) and \(r > 0\) we have
	\[{}
	\abs{\mathcal{F}K(\xi)}
	\le \norm{\mathcal{F}(K\psi_{r})}_{L^{\infty}(\R^{N})}
	+ \frac{1}{4 \pi^{2} \abs{\xi}^{2}} \norm{\mathcal{F}(\Delta K_{\infty, r})}_{L^{\infty}(\R^{N})}.
	\]
	Since
	\[{}
	\norm{\mathcal{F}(K\psi_{r})}_{L^{\infty}(\R^{N})}
	\le \norm{K\psi_{r}}_{L^{1}(\R^{N})}
	\le A \int_{B_{2r}}{\frac{\dif x}{\abs{x}^{N - 1}}}
	= \Cl{cte-502} r
	\]
	and
	\[{}
	\norm{\mathcal{F}(\Delta K_{\infty, r})}_{L^{\infty}(\R^{N})}
	\le \norm{\Delta K_{\infty, r}}_{L^{1}(\R^{N})}
	\le \Cr{cte-479} \int_{\R^{N} \setminus B_{r}}{\frac{\dif x}{\abs{x}^{N + 1}}}
	= \frac{\Cl{cte-509}}{r},
	\]
	we get
	\[{}
	\abs{\mathcal{F}K(\xi)}
	\le \Cr{cte-502} r + \frac{\Cr{cte-509}}{4 \pi^{2} \abs{\xi}^{2} \, r }.
	\]
	To conclude it thus suffices to take \(r = 1/\abs{\xi}\).
\end{proof}

\resetconstant
\begin{proof}[Proof of \Cref{L2caseNew}]
By a standard property of the Fourier transform of the convolution,
\[{}
\mathcal{F}(K * \mu)
= (\mathcal{F}K) (\mathcal{F}\mu).
\]
For \(j \in \{1, \dots, N\}\) and \(\varphi \in C_{c}^{\infty}(\R^{N})\), by the  Plancherel theorem we have
\[{}
\int_{\R^{N}}{K * \mu \, \frac{\partial\varphi}{\partial x_{j}}}
= \int_{\R^{N}}{\mathcal{F}(K * \mu) \, \overline{\mathcal{F}\Bigl(\frac{\partial\varphi}{\partial x_{j}} \Bigr)} }
= - \int_{\R^{N}}{(\mathcal{F}K) (\mathcal{F}\mu) \, 2 \pi \imath \xi_{j} \overline{\mathcal{F}\varphi} }.
\]
By \Cref{lemmaKHomogeneous}, the function \(\xi \mapsto \xi_{j}\mathcal{F}K(\xi)\) is bounded in \(\R^{N}\).
It thus follows from the Cauchy-Schwarz inequality and another application of the Plancherel theorem that
\[{}
\biggabs{\int_{\R^{N}}{K * \mu \, \frac{\partial\varphi}{\partial x_{j}}}}
\le 2 \pi C' \norm{\mathcal{F}\mu}_{L^{2}(\R^{N})}\norm{\mathcal{F}\varphi}_{L^{2}(\R^{N})}
= 2 \pi C' \norm{\mu}_{L^{2}(\R^{N})}\norm{\varphi}_{L^{2}(\R^{N})}.
\]
It now suffices to apply the Riesz representation theorem in \(L^{2}(\R^{N})\) to conclude.
\end{proof}

\section{Weak-$L^{1}$ estimate of the approximate derivative}\label{4}

The fundamental tool to prove \Cref{th:main1} is the weak-\(L^{1}\) estimate from the Calder\'on-Zygmund theory of singular integrals.
We revisit their approach to reformulate and make more transparent the role of the approximate differentiability as follows:

\begin{theorem}\label{theoremADSingularIntegral}
	If \(K \in C^{2}(\R^{N} \setminus \{0\})\) satisfies \eqref{eqKernel} and \(\mu\) is a finite Borel measure in $\R^N$, then \(K * \mu\) is approximately differentiable almost everywhere in \(\R^{N}\) and we have
	\begin{equation*}
	\seminorm{\apD(K * \mu)}_{L^{1}\lw(\R^{N})}
	\le C \norm{\mu}_{\cM(\R^{N})},
	\end{equation*}
	where \(C > 0\) depends on \(A\), \(B\) and \(N\).
\end{theorem}

\Cref{theoremADSingularIntegral} relies on the following estimate, whose proof we postpone to the next section: 

\begin{proposition}
	\label{propositionDifficultCase}
Let \(K \in C^{2}(\R^{N} \setminus \{0\})\) be a function that satisfies \eqref{eqKernel} and let \(Q \subset \R^{N}\) be a cube.
If $\nu$ is a finite Borel measure in $\R^N$ with 
\[{}
\nu(Q)=0{}
\quad \text{and} \quad{}
\abs{\nu}(\R^{N} \setminus Q) = 0,
\]
then for every $\theta>1$ there exists a bounded continuous function $I: \R^N \setminus \theta Q \rightarrow {[}0, \infty{)}$ such that
\[{}
|K \ast \nu(x)-K\ast \nu(y)| \leq \left( I(x)+I(y) \right) |x-y|
\quad \text{for every $x,y \in \R^N \setminus \theta Q$,}
\]
and
\[{}
\norm{I}_{L^{1}(\R^N\setminus \theta Q)} 
\leq C' 
\|\nu\|_{\M(\R^N)},
\]
where \(C' > 0\) depends on \(A\), \(B\), \(N\) and \(\theta\), but not on \(Q\).
\end{proposition}

We denote by \(\theta Q\) the rescaled cube with the same center as \(Q\) and side-length \(\theta\) times the side-length of \(Q\).{}
\Cref{propositionDifficultCase} has been proved by Haj\l asz without an explicit \(L^{1}\)-estimate of \(I\); see Lemma~9 in \cite{Hajlasz1996}. 
We rely on a variant of his proof that keeps track of the quantity \(\norm{I}_{L^{1}(\R^{N} \setminus \theta Q)}\). 

\resetconstant
\begin{proof}[Proof of \Cref{theoremADSingularIntegral}]
	Given \(t > 0\), we explain below that, thanks to \Cref{appdiff}, it is sufficient to prove the existence of
	a measurable function \(I : \R^{N} \to [0, \infty]\), 
	\emph{possibly depending on \(t\)}, such that
	\begin{equation}
	\label{eqSingularIntegralLipschitz}
	\abs{K * \mu(x) - K * \mu(y)}
	\le (I(x) + I(y)) \abs{x - y}
	\quad \text{for every \(x, y \in \Dom{(K * \mu)}\)}
	\end{equation}
	and
	\begin{equation}
	\label{eqSingularIntegralMeasure}
	\abs{\{I > t\}}
	\le \frac{\Cl{cte-647}}{t} \norm{\mu}_{\cM(\R^{N})},
	\end{equation}
	where \(\Cr{cte-647} > 0\) is independent of \(t\).{}
	In contrast with \Cref{theoremADSingularIntegralAlternative}, the set \(\{I = \infty\}\) need not be Lebesgue-negligible.
	We later show in \Cref{4'} that \(I\) can be chosen independently of \(t\) and then in this case \(\{I = \infty\}\) is negligible.
	For the sake of the proof of \Cref{theoremADSingularIntegral} such an independence of \(t\) is not needed.

	By analogy with the definition of the maximal function in the \(L^{1}\) setting, we first define the maximal function \(\cM\mu(x)\) of the Borel measure \(\mu\) in \(\R^{N}\) by computing the supremum of \(\abs{\mu}(B_{r}(x))/|B_{r}(x)|\) over \(r > 0\).{}
	The set
	\[{}
	F \vcentcolon= \{\mathcal{M}\mu \le t\}
	\]
	is closed and its complement verifies the weak maximal inequality, see p.~19 in \cite{Stein1970}:
	\begin{equation}
		\label{eqWeakMaximalInequality}
	\abs{\R^{N} \setminus F}
	= \abs{\{\mathcal{M}\mu > t\}}
	\le \frac{\Cl{cte-635}}{t} \norm{\mu}_{\cM(\R^{N})}.
	\end{equation}
		
	To construct a function \(I\) that satisfies \eqref{eqSingularIntegralLipschitz} and \eqref{eqSingularIntegralMeasure}, we take a Whitney covering of the open set \(\R^{N} \setminus F\) in terms of cubes \((Q_{n})_{n \in \N}\).{}
	We assume that each cube \(Q_{n}\) is \emph{half-closed}, by this we mean that \(Q_{n}\) is a Cartesian product of intervals of the form \([a_{i}, b_{i})\) with \(a_{i} < b_{i}\).{}
	Such a choice does not change Whitney's construction and yields a \emph{disjoint} family \((Q_{n})_{n \in \N}\).
	Each cube \(Q_{n}\) satisfies 
	\[
	2Q_{n} \subset \R^{N} \setminus F
	\quad \text{and} \quad
	\alpha Q_{n} \cap F \neq \emptyset,{}
	\]
	for some fixed \(\alpha > 2\), depending on \(N\). 
	Hence, 
	\begin{equation}
	\label{eqInclusionF}
	F \subset {\bigcap_{n = 0}^{\infty}}{\bigl(\R^{N} \setminus (2Q_{n}) \bigr)}
	\end{equation}
	and	there exist $0 < c_1 <c_2$ such that 
	\[{}
	c_1\diam(Q_n) \leq d(Q_n,F) \leq c_2\diam(Q_n).
	\]
	By the choice of \(F\), we also have
	\begin{equation}
	\label{eqMeasureCubes}
	\frac{\abs{\mu}(Q_{n})}{\abs{Q_{n}}} 
	\le \Cl{cte-655} t , 
	\end{equation}
	for some constant \(\Cr{cte-655} > 0\) depending on \(N\).{}
	Indeed, taking \(z \in \alpha Q_{n} \cap F\), we have \(\alpha Q_{n} \subset B_{r_{n}}(z)\), where \(r_{n} = \sqrt{N}\alpha\ell_{n}\)  and \(\ell_{n}\) is the side-length of \(Q_{n}\).{}
	Since \(\cM\mu(z) \le t\) and the volumes of \(\alpha Q_{n}\) and \(B_{r_{n}}(z)\) are comparable, we thus have	by monotonicity of \(\abs{\mu}\),
	\[{}
	\frac{\abs{\mu}(Q_{n})}{c\abs{Q_{n}}}
	\le \frac{\abs{\mu}(B_{r_{n}}(z))}{\abs{B_{r_{n}}(z)}}
	\le \cM\mu(z){}
	\le t,
	\]
	for a constant \(c > 0\) depending on \(\alpha\) and \(N\).{}
	
	Using the Whitney covering of \(\R^{N} \setminus F\), we now decompose the measure \(\mu\) as
	\[{}
	\mu 
	= \mu\lfloor_{F}{} + \mu\lfloor_{\R^{N} \setminus F}{}
	= \mu\lfloor_{F}{} + \sum_{n = 0}^{\infty}{\mu\lfloor_{Q_{n}}},
	\]
	which we further write as
	\begin{equation}
	\label{eqDecompositionCZThm}
	\mu = \underbrace{\mu\lfloor_{F}{} + \sum_{n = 0}^{\infty}{\frac{\mu(Q_{n})}{\abs{Q_{n}}} \chi_{Q_{n}} \dif x}}_{g \dif x}  + \underbrace{\sum_{n = 0}^{\infty}{b_{n}}}_{b}, 
	\end{equation}
	where 
	\[{}
	b_{n} \vcentcolon= \mu\lfloor_{Q_{n}}{} - \frac{\mu(Q_{n})}{\abs{Q_{n}}} \, \chi_{Q_{n}} \dif x,
	\]
	so that \(b_{n}(Q_{n}) = 0\) and \(|b_{n}|(\R^{N} \setminus Q_{n}) = 0\) for every \(n \in \N\).
	We refer the reader to the excellent introductions~\cites{Stein1970,Grafakos} for a detailed explanation of the Whitney covering of a set and the subsequent Calderón-Zygmund decomposition of a function or a measure.
	
	The term denoted by \(g \dif x\) in \eqref{eqDecompositionCZThm} is the good part of the measure and is absolutely continuous with respect to the Lebesgue measure.
	The absolute continuity of \(\mu\lfloor_{F}\) follows from the definition of \(F\), which implies that \(\abs{\mu}\lfloor_{F}(A) \le t \abs{A}\) 
	for every Borel subset \(A \subset \R^{N}\).
	By the Radon-Nikodym theorem, \(\mu\lfloor_{F}\) can be written as \(\mu\lfloor_{F}{} = f \dif x\) with \(f \in L^{1}(\R^{N})\) such that \(\abs{f} \le t\).{}
	
	We now observe that the density \(g\) belongs to \((L^{1} \cap L^{\infty})(\R^{N})\) and satisfies
	\begin{equation}
		\label{eqGoodPart}
	\norm{g}_{L^{1}(\R^{N})}
	\le \norm{\mu}_{\cM(\R^{N})}
	\quad{}
	\text{and}
	\quad{}
	\norm{g}_{L^{\infty}(\R^{N})}
	\le \Cl{cte-774} t.
	\end{equation}
	Indeed, we have \(\abs{g} = \abs{f} \le t\) on \(F\) and, by \eqref{eqMeasureCubes}, \(\abs{g} \le \Cr{cte-655} t \) on each \(Q_{n}\).
	Thus, the \(L^{\infty}\) bound of \(g\) holds with \(\Cr{cte-774} \vcentcolon= \max{\{1, \Cr{cte-655}\}}\).
	Since the cubes \(Q_{n}\) are disjoint, by additivity of \(\abs{\mu}\) we also have
	\[{}
	\norm{g}_{L^{1}(\R^{N})}
	= \abs{\mu}(F) + \sum_{n = 0}^{\infty}{\abs{\mu(Q_{n})}}
	\le \abs{\mu}(\R^{N})
	= \norm{\mu}_{\cM(\R^{N})}. 
	\]
	Hence, \eqref{eqGoodPart} is satisfied. 
	Since \(K \in (L^{1} + L^{\infty})(\R^{N})\) and \(g \in (L^{1} \cap L^{\infty})(\R^{N})\), we have
	\[{}
	\Dom{(K * g)} = \R^{N}
	\]
	and	the convolution \(K * g\) is continuous in \(\R^{N}\).
	
	As \(g \in (L^{1} \cap L^{\infty})(\R^{N})\), by interpolation between \(L^{1}\) and \(L^{\infty}\), we have \(g \in L^{2}(\R^{N})\). 
	We deduce from \Cref{L2caseNew} that \(\nabla(K * g) \in L^{2}(\R^{N})\).
	By  \Cref{MaximalandSobolev} and continuity of \(K * g\), we then have
	\begin{equation}
	\label{eqGoodLipschitz}
	\abs{K * g(x) - K * g(y)}
	\le (J(x) + J(y)) \abs{x - y}
	\quad
	\text{for every \(x, y \in \R^{N}\),}
	\end{equation}
	with \(J = 2^{N} \cM{\abs{\nabla(K * g)}}\).{}
	
	We now focus on \(b\), which is the bad part of the measure \(\mu\):
	\begin{Claim}
		There exists a Lebesgue-negligible set \(S \subset \R^{N}\) such that \(\Dom{(K * b)} \supset \R^{N} \setminus S\) and 
	\begin{equation}
		\label{eqDecompositionCZPointwiseBad}
		K * b(x){}
		= \sum_{n = 0}^{\infty}{K * b_{n}(x)}
		\quad \text{for  every \(x \in \R^{N} \setminus S\).} 
	\end{equation}
	\end{Claim}

	\begin{proof}[Proof of the Claim]
	Let
	\begin{equation}
		\label{eqDecompositionSetS}
		S \vcentcolon= \biggl\{ z \in \R^{N} : \sum_{n = 0}^{\infty}{\abs{K} * \abs{b_{n}}(z)} = \infty \biggr\}.
	\end{equation}
	For \(x \in \R^{N} \setminus S\), we have \(\sum\limits_{n = 0}^{\infty}{\abs{K} * \abs{b_{n}}(x)} < \infty\).{}
	Then, by Fatou's lemma, \(\abs{K}*\abs{b}(x) < \infty\) and so
	\[{}
	\Dom{(K * b)} \supset \R^{N} \setminus S.
	\] 
	The Dominated convergence theorem implies \eqref{eqDecompositionCZPointwiseBad}.
	To prove that \(S\) is negligible, we proceed as follows.
	By Fubini's theorem and assumption \eqref{eqKernel}, for every \(r > 0\) and \(n \in \N\) we have
	\[{}
	\begin{split}
	\int_{B_{r}(0)}{\abs{K} * \abs{b_{n}}}
	& = \int_{\R^{N}}{\biggl( \int_{B_{r}(0)}{\abs{K(x - y)} \dif x} \biggr)} \dif\abs{b_{n}}(y)\\
	& \le A \int_{\R^{N}}{\biggl( \int_{B_{r}(0)}{\frac{\dif x}{\abs{x - y}^{N-1}}} \biggr)} \dif\abs{b_{n}}(y).
	\end{split}
	\]
	For every \(y \in \R^{N}\),{}
	\[{}
	\int_{B_{r}(0)}{\frac{\dif x}{\abs{x - y}^{N-1}}}
	\le \int_{B_{r}(y)}{\frac{\dif x}{\abs{x - y}^{N-1}}}
	= \int_{B_{r}(0)}{\frac{\dif z}{\abs{z}^{N-1}}}
	= \Cl{cte-996} r.
	\]
	Thus,
	\[
	\int_{B_{r}(0)}{\abs{K} * \abs{b_{n}}}
	\le A \Cr{cte-996}r \int_{\R^{N}}{\dif\abs{b_{n}}}
	= \Cl{cte-1001} \abs{b_{n}}(\R^{N})
	\le 2\Cr{cte-1001} \abs{\mu}(Q_{n}),
	\]
	for some constant \(\Cr{cte-1001} > 0\) depending on \(r\).
	Since the cubes \(Q_{n}\) are disjoint, by Fatou's lemma and by additivity of the measure \(\abs{\mu}\) we then get
	\[{}
	\begin{split}
	\int_{B_{r}(0)}{\biggl(\,\sum_{n = 0}^{\infty}{\abs{K} * \abs{b_{n}}} \biggr)}
	&\le {\sum_{n = 0}^{\infty}{\int_{B_{r}(0)} \abs{K} * \abs{b_{n}}}}\\
	&\le 2\Cr{cte-1001} \sum_{n = 0}^{\infty}{\abs{\mu}(Q_{n})}
	= 2\Cr{cte-1001} \abs{\mu}(\R^{N} \setminus F)
	< \infty.{}
	\end{split}
	\]
	Hence, \(\sum\limits_{n = 0}^{\infty}{\abs{K} * \abs{b_{n}}} < \infty\) almost everywhere in \(B_{r}(0)\), for every \(r > 0\). 
	We conclude that \(S\) is Lebesgue-negligible.{}
	\end{proof}
	
	As a consequence of the Claim,
	\[{}
	\Dom{(K * \mu)} \supset \Dom{(K * g)} \cap \Dom{(K * b)} \supset \R^{N} \setminus S
	\]
	and, from linearity of the convolution,
	\begin{equation}
		\label{eqDecompositionCZPointwise}
		K*\mu(x){}
		= K*g(x) + \sum_{n = 0}^{\infty}{K*b_{n}(x)}
		\quad \text{for every \(x \in \R^{N} \setminus S\).} 
	\end{equation}

	By \Cref{propositionDifficultCase} with \(\theta = 2\), each measure \(b_{n}\) satisfies
	\[{}
	\abs{K * b_{n}(x) - K * b_{n}(y)}
	\le (I_{n}(x) + I_{n}(y)) \abs{x - y}
	\quad \text{for every \(x, y \in \R^{N} \setminus 2Q_{n}\),}
	\]
	where \(I_{n} : \R^{N} \setminus 2Q_{n} \to {}[0, \infty)\) is a bounded continuous function such that
	\begin{equation}
		\label{eqEstimateIn}
	\norm{I_{n}}_{L^{1}(\R^{N} \setminus 2Q_{n})}
	\le C' \abs{b_{n}}(Q_{n})
	\le 2 C' \abs{\mu}(Q_{n}).
	\end{equation}
	By \eqref{eqInclusionF} and \eqref{eqDecompositionCZPointwiseBad}, we thus have
	\begin{equation}
		\label{eqBadLipschitz}
	\abs{K * b(x) - K * b(y)}
	\le \biggl(\,\sum_{n=0}^{\infty}{I_{n}(x)} + \sum_{n=0}^{\infty}{I_{n}(y)}\biggr) \abs{x - y}
	\quad \text{for every \(x, y \in F \setminus S\).}
	\end{equation}
	
	Combining \eqref{eqGoodLipschitz} and \eqref{eqBadLipschitz}, we get \eqref{eqSingularIntegralLipschitz} with
	\[{}
	I \vcentcolon= 
	\begin{cases}
		J + \sum\limits_{n = 0}^{\infty}{I_{n}}
		& \text{in \(F \setminus S\),}\\
		\infty{}
		& \text{in \((\R^{N} \setminus F) \cup S\).}{}
	\end{cases}
	\]
	To prove \eqref{eqSingularIntegralMeasure}, we first observe that by subadditivity of the Lebesgue measure,
	\begin{equation}
		\label{eqEstimateLevelSets}
	\abs{\{I > t\}}
	\le \abs{\{J > t/2\}} + \Bigabs{\Bigl\{\sum_{n=0}^{\infty}{I_{n}} > t/2 \Bigr\} \cap F} + \abs{\R^{N} \setminus F} + \abs{S}.
	\end{equation}
	As \(S\) is Lebesgue-negligible, \(\abs{S} = 0\).
	Since \(J = 2^{N} \cM{\abs{\nabla(K * g)}}\), by the \(L^{2}\)-maximal inequality and \Cref{L2caseNew} we have
	\[
	\norm{J}_{L^{2}(\R^{N})}
	\le \C \norm{\nabla(K * g)}_{L^{2}(\R^{N})}
	\le \Cl{cte-701} \norm{g}_{L^{2}(\R^{N})}.
	\]
	We also have by \eqref{eqGoodPart} and interpolation between Lebesgue spaces,
	\[{}
	\norm{g}_{L^{2}(\R^{N})}^{2}
	\le \norm{g}_{L^{1}(\R^{N})} \norm{g}_{L^{\infty}(\R^{N})}  
	\le \Cr{cte-774} \, t \norm{\mu}_{\cM(\R^{N})}.
	\]
	By the Markov-Chebyshev inequality at height \(t/2\), we thus get
	\begin{equation}
		\label{eqSingularIntegral1}
	\abs{\{J > t/2\}} 
	\le \Bigl(\frac{2}{t}\Bigr)^{2} \norm{J}_{L^{2}(\R^{N})}^{2}
	\le \frac{\C}{t} \norm{\mu}_{\cM(\R^{N})},
	\end{equation}
	which gives the estimate of the first term in the right-hand side of \eqref{eqEstimateLevelSets}.

	By \eqref{eqEstimateIn} and the fact that the cubes \(Q_{n}\) are disjoint and contained in \(\R^{N} \setminus F\),
	\begin{equation}
	\label{eqEstimateMarcinkiewicz}
	\Bignorm{\sum_{n=0}^{\infty}{I_{n}}}_{L^{1}(F)}
	 \le \sum_{n=0}^{\infty}{\norm{I_{n}}_{L^{1}(\R^{N} \setminus 2Q_{n})}}
	 \le 2 C' \sum_{n=0}^{\infty}{\abs{\mu}(Q_{n})}
	 \le 2 C' \abs{\mu}(\R^{N} \setminus F),
	\end{equation}
	which can also be deduced from a classical inequality in Harmonic analysis; see \Cref{remarkMarcinkiewicz} below.	
	By the Markov-Chebyshev inequality at height \(t/2\), we then have
	\begin{equation}
		\label{eqSingularIntegral3}
	\Bigabs{\Bigl\{\sum_{n=0}^{\infty}{I_{n}} > t/2 \Bigr\} \cap F} 
	\le \frac{4 C'}{t} \abs{\mu}(\R^{N} \setminus F) 
	\le \frac{4 C'}{t} \norm{\mu}_{\cM(\R^{N})}.
	\end{equation}
	Inequality \eqref{eqSingularIntegralMeasure} then follows from \eqref{eqEstimateLevelSets}, \eqref{eqSingularIntegral1}, \eqref{eqSingularIntegral3} and  \eqref{eqWeakMaximalInequality}.

	Observe that \(K * \mu\) is defined in 
	\[{}
	\{I_{t} < \infty\} \subset \R^{N} \setminus S,
	\]
	where for the rest of proof we make explicit the dependence of \(I = I_{t}\) with respect to the parameter \(t\).{}
	We conclude from \eqref{eqSingularIntegralLipschitz} and \Cref{appdiff} that \(K * \mu\) is approximately differentiable almost everywhere in \(\{I_{t} < \infty\}\).
	Taking a sequence \((t_{j})_{j \in \N}\) of positive numbers with \(t_{j} \to \infty\), we have
	\[{}
	\R^{N} \setminus \bigcup\limits_{j = 0}^{\infty}{\{I_{t_{j}} < \infty\}}
	\subset \bigcap_{j = 0}^{\infty}{\{I_{t_{j}} > t_{j}\}}
	\] 
	As a consequence of \eqref{eqSingularIntegralMeasure}, the set in the right-hand side is Lebesgue-negligible and we deduce that \(K * \mu\) is approximately differentiable almost everywhere in \(\R^{N}\).{}
	Since the approximate derivative \(\apD(K * \mu)\) satisfies
	\[{}
	\abs{\apD{(K * \mu)}} 
	\le 2 t
	\quad \text{almost everywhere in \(\{I_{t} \le t\}\),}
	\]
	we have
	\[{}
	\bigabs{\{\abs{\apD{(K * \mu)}} > 2t\}}
	 \le \abs{\{I_{t} > t\}} \le \frac{\Cr{cte-647}}{t} \norm{\mu}_{\cM(\R^{N})}
	 \quad \text{for every \(t > 0\)}.
	\]
	This implies the weak-\(L^{1}\) estimate of \(\apD{(K * \mu)}\).
\end{proof}

\resetconstant
\begin{remark}
	\label{remarkMarcinkiewicz}	
	Estimate~\eqref{eqEstimateMarcinkiewicz} also follows from the classical inequality satisfied by the Marcinkiewicz integral, see p.~15 of \cite{Stein1970}:
	\begin{equation}
	\label{eqMarcinkiewiczRemark}
	\int_{F}{\int_{\R^{N} \setminus F}{\frac{d(z, F)}{\abs{y - z}^{N + 1}} \dif\abs{\mu}(z)} \dif y}
	\le C'' \abs{\mu}(\R^{N} \setminus F).
	\end{equation}
	The reason is that in the proof of \Cref{propositionDifficultCase}, see \eqref{eqChoiceI}, we choose
	\[{}
	I_{n}(y) = \frac{\Cl{cte-899} \ell_{n}}{\abs{y - \bar{z}_{n}}^{N+1}} \norm{b_{n}}_{\cM(\R^{N})}
	\quad \text{for every \(y \in \R^{N} \setminus 2Q_{n}\),}
	\]
	where \(\bar{z}_{n}\) is the center of the cube \(Q_{n}\), and this function \(I_{n}\) is controlled by the Marcinkiewicz integral of the measure \(\abs{\mu}\lfloor_{Q_{n}}\), namely 
	\[{}
	\int_{Q_n} \frac{d(z, F)}{|y-z|^{N+1}}d|\mu|(z).{}
	\]
	Indeed, by construction of the Whitney covering, for every $z\in Q_n$ we have $\ell_n/2 \leq d(z, F)$ and, for every \(y \in \R^{N} \setminus 2Q_{n}\),
	\[{}
	|y - z|
	\le |y - \bar{z}_n| + |z - \bar{z}_n|
	\le |y - \bar{z}_n| + \sqrt{N} \, \frac{\ell_n}{2} 
	\le \biggl( 1 + \frac{\sqrt{N}}{2} \biggr) |y - \bar{z}_n|.
	\]  
	Thus,
	\begin{equation}
	\label{eqComparisonMarcinkiewicz}
	I_{n}(y)
	\le \frac{\Cr{cte-899} \ell_{n}}{\abs{y - \bar{z}_{n}}^{N+1}} \, 2 \int_{Q_{n}}{\dif\abs{\mu}(z)}
	\le \Cl{cte-1105} \int_{Q_n} \frac{d(z, F)}{|y-z|^{N+1}}d|\mu|(z).
	\end{equation}
	Therefore, for every \(y \in F\),
	\[{}
	\sum_{n = 0}^{\infty}{I_{n}(y)}
	\le \Cr{cte-1105} {\int_{\bigcup\limits_{n=0}^{\infty}{Q_n}} \frac{d(z, F)}{|y-z|^{N+1}}d|\mu|(z)}
	= \Cr{cte-1105} {\int_{\R^{N} \setminus F} \frac{d(z, F)}{|y-z|^{N+1}}d|\mu|(z)}
	\]
	and then \eqref{eqEstimateMarcinkiewicz} is implied by \eqref{eqMarcinkiewiczRemark}.
	We shall return to this observation in the proof of \Cref{theoremDifferentiability}.
\end{remark}

\section{Proof of \Cref{propositionDifficultCase}}
\label{sectionProofProposition}

In the next two lemmas we rely on the notation of \Cref{propositionDifficultCase} 
and, in particular, \(\nu\) is a finite Borel measure in \(\R^{N}\) with
\[{}
\nu(Q) = 0
\quad \text{and} \quad{}
\abs{\nu}(\R^{N} \setminus Q) = 0.
\]
We also write \(\bar{z}\) for the center of the cube \(Q\) and \(\ell\) for its side-length.
The first lemma gives an estimate of the decay of \(K * \nu(x)\) as \(\abs{x} \to \infty\) and is used in the proof of \Cref{propositionDifficultCase} when \(x\) and \(y\) are far from each other relatively to their distances to \(Q\).

\begin{lemma}\label{lemmaDecayEstimateFar}
For every $x \in \R^N\setminus \theta Q$,
\begin{equation*}
\abs{K\ast \nu (x)}\leq \frac{C''\ell}{|x-\bar{z}|^{N}}\|\nu\|_{\M(\R^N)}.
\end{equation*}
\end{lemma}

\resetconstant
\begin{proof}
Since $\nu(Q)=0$, we can write 
\begin{equation}
	\label{eqFarLipschitz}
K \ast \nu(x) 
=\int_Q K(x-z) \dif \nu(z)
= \int_Q [K(x-z)-K(x-\bar{z})] \dif \nu(z). 
\end{equation}
The estimate of the integrand that we require is given by 

\begin{Claim}
	For every \(z \in Q\) and \(x \in \R^N \setminus \theta Q\),{}
	\[{}
	\abs{K(x-z)-K(x-\bar{z})}
	\le \Cl{cte-775} \frac{\abs{z - \bar{z}}}{\abs{x - \bar{z}}^{N}},
	\]
	where \(\Cr{cte-775} > 0\) depends on \(\theta\), \(A\), \(B\) and \(N\).
\end{Claim}

\begin{proof}[Proof of the Claim]
	By the Mean value theorem, there exists a point \(\zeta\) in the line segment $[z, \bar{z}]$ that joins \(z\) and \(\bar{z}\) such that
	\[{}
	K(x-z)-K(x-\bar{z})
	= \nabla K(x - \zeta) \cdot (z - \bar{z}).
	\]
	By \eqref{eqKernel1}, we thus have
	\begin{equation}
		\label{eqFarMeanValue}
	\abs{K(x-z)-K(x-\bar{z})}
	\le \C \frac{\abs{z - \bar{z}}}{\abs{x - \zeta}^{N}}.
	\end{equation}
	Since \(\zeta \in Q\), we have \(\abs{\zeta - \bar{z}}_{\infty} \le \ell/2\), where \(\abs{y}_{\infty}\) denotes the max-norm of a vector \(y = (y_{1}, \dots, y_{N})\) in \(\R^{N}\), i.e.
\[{}
\abs{y}_{\infty}
\vcentcolon= \max{\{\abs{y_{1}}, \dots, \abs{y_{N}}\}}.
\]
	For $x \in \R^N \setminus \theta Q$, we also have \(\abs{x - \bar{z}}_{\infty} \ge {\theta\ell}/{2}\).{}
	Thus, by the triangle inequality,
	\[
	|x-\bar{z}|_{\infty} 
	\leq | x- \zeta|_{\infty} +|\zeta - \bar{z}|_{\infty}
	\leq |x-\zeta|_{\infty} + \frac{\ell}{2}
	\leq |x-\zeta|_{\infty} + \frac{1}{\theta} |x-\bar{z}|_{\infty}.
	\]
	This yields 
	\begin{equation*}
	\Bigl(1 - \frac{1}{\theta} \Bigr) |x-\bar{z}|_{\infty} 
	\leq |x-\zeta|_{\infty}
	\end{equation*}
	and then a similar estimate is satisfied by the Euclidean norm.
	The Claim thus follows from such an estimate and \eqref{eqFarMeanValue}.
\end{proof}

From \eqref{eqFarLipschitz} and the Claim, we then get
\begin{equation*}
\abs{ K \ast \nu (x) } 
\leq  \frac{\Cl{cte-832}\ell}{|x-\bar{z}|^N} \int_{Q} \dif|\nu|(z)
= \frac{\Cr{cte-832} \ell}{|x-\bar{z}|^N} \, \norm{\nu}_{\cM(\R^{N})}.
\qedhere
\end{equation*}
\end{proof}

The next lemma deals with the case where \(x\) and \(y\) are close to each other relatively to their distances to \(Q\):

\begin{lemma}\label{lemmaDecayEstimateClose}
There exists \(\epsilon > 0\), depending on \(\theta\) and \(N\), such that for every \(x, y \in \R^N\setminus \theta Q\) with \(\abs{x - y} \le \epsilon\abs{x - \bar{z}}\), we have
\[
\abs{K\ast \nu (x) - K\ast \nu (y)}
\leq \frac{C''' \ell}{|x - \bar{z}|^{N+1}}  \|\nu\|_{\M(\R^N)} \, \abs{x - y}.
\]
In particular,  for every \(x \in \R^N\setminus \theta Q\),
\[{}
\abs{\nabla(K * \nu)(x)}
\le \frac{C''' \ell}{|x - \bar{z}|^{N+1}}  \|\nu\|_{\M(\R^N)}.
\]
\end{lemma}

\resetconstant
\begin{proof}
Applying the Fundamental theorem of Calculus and Fubini's theorem, we have
\[{}
\begin{split}
K \ast \nu(x)-K \ast \nu(y){}
& = \int_0^1\nabla (K\ast\nu)(\xi_t) \cdot (x-y)\dif t\\
& = \int_0^1{\int_Q \nabla K(\xi_{t} - z) \cdot (x-y) \dif\nu(z)}\dif t,
\end{split}
\]
where \(\xi_{t} \vcentcolon= tx+(1-t)y\) belongs to the line segment \([x, y]\) for every \(t \in [0, 1]\).{}
We do not explicit its dependence on \(x\) and \(y\), but one should keep in mind that \emph{\(\xi_{t}\) is independent of \(z\)}.
Since \(\nu(Q) = 0\), we can thus write
\begin{equation}
\label{eqBadMeanValue}
K \ast \nu(x)-K \ast \nu(y){} 
=\int_0^1{ \int_Q{ \bigl[\nabla K(\xi_{t} - z) - \nabla K(\xi_{t} - \bar{z})\bigr] \cdot (x-y) \dif \nu(z)} \dif t}.
\end{equation}

\begin{Claim}
	Let \(0 < \beta < \frac{\theta -1}{\theta}\) and  \(x \in \R^N\setminus \theta Q\).{}
	For every \(\xi \in \R^N\) such that \(\abs{x - \xi}_{\infty} \le \beta \abs{x - \bar{z}}_{\infty}\) and every \(z \in Q\), we have
	\[{}
	\bigabs{\nabla K(\xi - z) - \nabla K(\xi - \bar{z})}
	\le \Cl{cte-860} \frac{\abs{z - \bar{z}}}{\abs{x - \bar{z}}^{N + 1}},
	\]
	where \(\Cr{cte-860} > 0\) depends on \(\beta\), \(\theta\), \(A\), \(B\) and \(N\).
\end{Claim}

\begin{proof}[Proof of the Claim]
	Applying the Mean value theorem, we deduce from the second estimate in \eqref{eqKernel} that there exists \(\zeta \in [z, \bar{z}]\) such that
	\begin{equation}
	\label{claim-Close}
	\bigabs{\nabla K(\xi - z) - \nabla K(\xi - \bar{z})}
	\le \C \frac{\abs{z - \bar{z}}}{\abs{\xi - \zeta}^{N + 1}}.
	\end{equation}
	We now show that for any \(\beta\) as above, one has
	\begin{equation}
	\label{claim-531}
	\abs{x - \bar{z}}_{\infty} 
	\le  \Cl{cte-496} \abs{\xi - \zeta}_{\infty}.
	\end{equation}
	To this end,
	observe that for \(\zeta \in Q\),{}
	\[{}
	\abs{\zeta - \bar{z}}_{\infty}
	\le \frac{\ell}{2}
	\le \frac{1}{\theta} \abs{x - \bar{z}}_{\infty}.
	\]
	By the triangle inequality and the assumption on \(\xi\), we thus have
	\[
	\abs{x - \bar{z}}_{\infty}
	\le \abs{x - \xi}_{\infty}
	+ \abs{\xi - \zeta}_{\infty}
	+ \abs{\zeta - \bar{z}}_{\infty}
	\le \Bigl(\beta + \frac{1}{\theta}\Bigr) \abs{x - \bar{z}}_{\infty} + \abs{\xi - \zeta}_{\infty}
	\]
	and we conclude that
	\[
	\Bigl( 1 - \beta - \frac{1}{\theta} \Bigr) \abs{x - \bar{z}}_{\infty}
	\le  \abs{\xi - \zeta}_{\infty}.
	\]
	Since the quantity in parenthesis is positive by the choice of \(\beta\), this inequality is equivalent to \eqref{claim-531}.
	The claim thus follows from \eqref{claim-Close} and the counterpart of \eqref{claim-531} for the Euclidean norm.
\end{proof}

To conclude the proof of the lemma, observe that if \(\abs{x - y} \le \epsilon \abs{x - z}\) for some  \(\epsilon > 0\), then 
\(\abs{x - y}_{\infty} 
\le \epsilon \sqrt{N} \abs{x - z}_{\infty}\).
Thus, taking \(\epsilon \sqrt{N} = \frac{\theta - 1}{2 \theta}\), every \(\xi \in [x, y]\) satisfies the assumption of the Claim.
By \eqref{eqBadMeanValue} and the Claim, we thus get
\[{}
\begin{split}
|K\ast \nu(x)-K \ast \nu(y)|{}
& \leq \frac{\Cl{cte-907} \ell}{\abs{x - \bar{z}}^{N+1}} \int_{Q}{\dif |\nu|(z)} \, |x-y|\\
& =  \frac{\Cr{cte-907} \ell}{\abs{x - \bar{z}}^{N+1}} \norm{\nu}_{\cM(\R^{N})} \, |x-y|,
\end{split}
\]
which gives the main estimate of the lemma.
From there, one estimates \(\nabla{(K * \nu)}(x)\) by letting \(y \to x\) in the direction of the gradient.
\end{proof}

\resetconstant
\begin{proof}[Proof of \Cref{propositionDifficultCase}]
	Let \(\epsilon > 0\) be as in \Cref{lemmaDecayEstimateClose}.
	Assuming that \(x, y \in \R^{N} \setminus \theta Q\) satisfy
	\[{}
	\abs{x - y}
	\le \epsilon \max{\{\abs{x - \bar{z}}, \abs{y - \bar{z}}\}},
	\]
	then after relabeling \(x\) and \(y\) if necessary we have
	\(
	\abs{x - y}
	\le \epsilon \abs{x - \bar{z}}
	\).
	By \Cref{lemmaDecayEstimateClose}, we thus have
	\[
	\abs{K\ast \nu (x) - K\ast \nu (y)}
	\leq C''' \ell \biggl( \frac{1}{|x - \bar{z}|^{N+1}} + \frac{1}{|y - \bar{z}|^{N+1}} \biggr) \|\nu\|_{\M(\R^N)} \, \abs{x - y}.
	\]
	
	We now assume instead that \(x, y \in \R^{N} \setminus \theta Q\) satisfy
	\[{}
	\epsilon \max{\{\abs{x - \bar{z}}, \abs{y - \bar{z}}\}} 
	< \abs{x - y}.
	\]
	The triangle inequality and \Cref{lemmaDecayEstimateFar} imply that
	\[{}
	\abs{K\ast \nu (x) - K\ast \nu (y)}
	\le C''\ell \biggl( \frac{1}{|x-\bar{z}|^{N}} + \frac{1}{|y-\bar{z}|^{N}} \biggr) \|\nu\|_{\M(\R^N)}.
	\]
	In view of the assumption on \(x\) and \(y\) in this case,
	\[{}
	\abs{K\ast \nu (x) - K\ast \nu (y)}
	\le C'' \ell \biggl( \frac{1}{|x-\bar{z}|^{N+1}} + \frac{1}{|y-\bar{z}|^{N + 1}} \biggr) \|\nu\|_{\M(\R^N)} \, \frac{\abs{x - y}}{\epsilon}.
	\]
	We thus have the conclusion with \(I : \R^{N} \setminus \theta Q \to {}[0, \infty)\) defined by
	\begin{equation}
	\label{eqChoiceI}
	I(y) = \frac{C' \ell}{|y - \bar{z}|^{N+1}} \|\nu\|_{\M(\R^N)},
	\end{equation}
	where \(C' \vcentcolon= \max{\{C''', C''/\epsilon\}}\).
\end{proof}

\section{A uniformization principle}\label{4'}

We now establish \Cref{theoremADSingularIntegralAlternative}, whose main ingredient is already contained in the proof of \Cref{theoremADSingularIntegral}.
There we show that, for every \(t > 0\), there exists a measurable function \(I = I_{t} : \R^{N} \to [0, \infty]\) which satisfies  \eqref{eqSingularIntegralLipschitz} and \eqref{eqSingularIntegralMeasure}.
	The distribution function of \(I_{t}\) only verifies the estimate we seek at height \(t\).
	The next lemma is a uniformization property that encodes this family of functions \((I_{t})_{t > 0}\) into a single weak-\(L^{1}\) function and immediately implies \Cref{theoremADSingularIntegralAlternative} by choosing 
	\[{}
	E = \Dom{(K * \mu)}
	\quad \text{and}\quad{}
	A' = \Cr{cte-647} \norm{\mu}_{\cM(\R^{N})},
	\]
	where \(\Cr{cte-647}\) is the constant in \eqref{eqSingularIntegralMeasure}.

\begin{lemma}
	Let \(E \subset \R^{N}\) and \(v : E \to \R\) be such that, for every \(t > 0\), there exists a measurable function \(I_{t} : \R^{N} \to [0, \infty]\) with
	\[{}
	\abs{v(x) - v(y)}
	\le (I_{t}(x) + I_{t}(y)) \abs{x - y}
	\quad \text{for every \(x, y \in E\)}
	\]
	and 
	\[{}
	\abs{\{I_{t} > t\}}
	\le \frac{A'}{t},
	\] 
	for some constant \(A' > 0\) independent of \(t\).{}
	Then, there exists a measurable function \(H : \R^{N} \to [0, \infty]\) with
	\begin{equation}\label{good_estimate}
	\abs{v(x) - v(y)}
	\le (H(x) + H(y)) \abs{x - y}
	\quad \text{for every \(x, y \in E\)}
	\end{equation}
	and  
	\[
	\seminorm{H}_{L^{1}\lw(\R^{N})}
	\le 8A'.
	\]
\end{lemma}

\begin{proof}
	Given \(x, y \in E\) such that \(v(x) \ne v(y)\), take \(n \in \Z\) such that
	\[{}
	2^{n+1} \abs{x - y}
	< \abs{v(x) - v(y)}
	\le 2^{n+2} \abs{x - y}.
	\]
	By the first inequality, for each \(t > 0\) we have that \(x\) or \(y\) belongs to \(\{I_{t} > 2^{n}\}\).{}
	By the second inequality, we thus have
	\begin{equation}
		\label{eqUniformatizationLevelN}
	\abs{v(x) - v(y)}
	\le 2^{n+2} \bigl(\chi_{\{I_{t} > 2^{n}\}}(x) + \chi_{\{I_{t} > 2^{n}\}}(y) \bigr) \abs{x - y}.
	\end{equation}

	Define \(H : \R^{N} \to [0, \infty]\) by
	\[{}
	H
	= \sup{\bigl\{2^{n+2} \, \chi_{\{I_{2^{n}} > 2^{n}\}} : n \in \Z\bigr\}}.
	\]
	Then, \eqref{good_estimate} holds.		
	To prove that \(H\)  satisfies the desired weak-\(L^{1}\) estimate, we first observe that, from the definition of \(H\),
	\begin{equation}
		\label{eqUniformizationInclusion}
		\{H > 2^{k}\} \subset \bigcup_{l = k-1}^{\infty}{\{I_{2^{l}} > 2^{l}\}}
		\quad \text{for every \(k \in \Z\).}
	\end{equation}
	Given \(t > 0\), let \(k \in \Z\) be such that 
	\(
	2^{k} \le t < 2^{k+1}
	\).
	By monotonicity and subadditivity of the Lebesgue measure, it follows from \eqref{eqUniformizationInclusion} that
	\[{}
	\abs{\{H > t\}}
	\le \abs{\{H > 2^{k}\}}
	\le \sum_{l = k-1}^{\infty}{\abs{\{I_{2^{l}} > 2^{l}\}}}.
	\]
	From the assumption on the measure of the superlevel sets of \(I_{t}\),
	\[{}
	\abs{\{H > t\}}
	\le \sum_{l = k-1}^{\infty}{\frac{A'}{2^{l}}}
	= \frac{A'}{2^{k-2}}
	\le \frac{8A'}{t}.
	\]
	Since \(t > 0\) is arbitrary, we thus have \(\seminorm{H}_{L^{1}\lw(\R^{N})}\le 8A' \).
\end{proof}

The classical \(L^{p}\) singular-integral estimates can be also formulated using a Lipschitz-type formalism, whose proof relies on a standard interpolation argument that we sketch for the convenience of the reader:

\begin{proposition}
	\label{remarkLp}
	Let \(K \in C^{2}(\R^{N} \setminus \{0\})\) be any function that satisfies \eqref{eqKernel} and let \(1 < p < \infty\).
	If \(\mu \in (L^{1} \cap L^{p})(\R^{N})\), 
	then \(K * \mu\) has a distributional gradient \(\nabla(K * \mu)\) in \(L^{p}(\R^{N}; \R^{N})\) and there exists a measurable function \(I : \R^{N} \to [0, \infty]\) in \(L^{p}(\R^{N})\) such that
	\begin{equation}
	\label{eqEstimateLipschitz}
	\abs{K * \mu(x) - K * \mu(y)}
	\le (I(x) + I(y)) \abs{x - y}
	\quad \text{for every \(x, y \in \Dom{(K * \mu)}\)}
	\end{equation}	
	and
	\begin{equation*}
	\norm{I}_{L^{p}(\R^{N})}
	\le C \norm{\mu}_{L^{p}(\R^{N})},
	\end{equation*}
	where the constant $C > 0$ depends on \(A\), \(B\), $p$ and $N$, 
\end{proposition}

\resetconstant
\begin{proof}
	A combination of \Cref{L2caseNew}, \Cref{theoremADSingularIntegral} and the Marcinkiewicz interpolation theorem implies that
the linear functional
\begin{equation}
	\label{eqFunctional}
\mu \in C_{c}^{\infty}(\R^{N})
\longmapsto \nabla(K * \mu) \in L^{p}(\R^{N})
\end{equation}
is continuous with respect to the strong \(L^{p}\)~topology on both sides for \(1 < p < 2\)
and then, by duality, the same conclusion holds for \(2 < p < \infty\).{}
By unique continuous extension of \eqref{eqFunctional} one deduces that the distributional derivative \(\nabla(K * \mu)\) belongs to \(L^{p}(\R^{N})\) for every \(\mu \in (L^{1} \cap L^{p})(\R^{N})\).{}

We claim that \eqref{eqEstimateLipschitz} holds with 
\begin{equation}
	\label{eqChoiceILp}
I \vcentcolon= 2^{N} \cM\abs{\nabla (K * \mu)} + R * \abs{\mu},
\end{equation}
where \(R(\xi) \vcentcolon= \chi_{B_{1}}(\xi)/\abs{\xi}^{N-1}\).{}
Observe that such an explicit choice of coefficient \(I\)\/ behaves sublinearly with respect to \(\mu\).

In view of \Cref{MaximalandSobolev}, to check the claim it suffices to verify that every \(x \in \R^{N}\) with \(I(x) < \infty\) is a Lebesgue point of \(K * \mu\) and \(K * \mu(x)\) is the precise representative.
That \(x\) is a Lebesgue point of \(K * \mu\) is a consequence of \(\cM\abs{\nabla (K * \mu)}(x) < \infty\), but without identification of the precise representative.{}
The full property can be obtained instead using \(R * \abs{\mu}(x) < \infty\), as it implies that
\begin{equation}
	\label{eqPreciseRepresentativeQ}
\lim_{r \to 0}{\fint_{B_{r}(x)}{\abs{K * \mu - K * \mu(x)}}} = 0.
\end{equation}
Indeed, by Fubini's theorem, for every \(r > 0\) we have
\[{}
\fint_{B_{r}(x)}{\abs{K * \mu - K * \mu(x)}}
\le \int_{\R^{N}}{\biggl( \fint_{B_{r}(x)}{\abs{K(y - z)- K(x - z)}\dif y} \biggr) \dif\abs{\mu}(z)}.
\]
By continuity of \(K\) in \(\R^{N} \setminus \{0\}\),{}
\[{}
\lim_{r \to 0}{\fint_{B_{r}(x)}{\abs{K(y - z)- K(x - z)}\dif y}} = 0
\quad \text{for every \(z \in \R^{N} \setminus \{x\}\).}
\]
Since \(\abs{K(\xi)} \le A/\abs{\xi}^{N - 1}\), we also have
\[{}
\begin{split}
\fint_{B_{r}(x)}{\abs{K(y - z)- K(x - z)} \dif y}
& \le \fint_{B_{r}(x)}{\abs{K(y - z)}} \dif y + \abs{K(x - z)}\\
& \le \fint_{B_{r}(x)}{\frac{A}{\abs{y - z}^{N - 1}}} \dif y + \frac{A}{\abs{x - z}^{N - 1}}\\
&\le \frac{\Cl{cte-1485}}{\abs{x - z}^{N - 1}}
 \le \Cr{cte-1485} \bigl(R(x - z) + 1\bigr),
\end{split}
\]
for every \(x, z \in \R^{N}\) and \(r > 0\), where the constant \(\Cr{cte-1485} > 0\) depends on \(A\) and \(N\).{}
As \(R * \abs{\mu}(x) < \infty\) and \(\mu \in L^{1}(\R^{N})\), we can apply the Dominated convergence theorem to deduce \eqref{eqPreciseRepresentativeQ}.
Hence, \(K*\mu(x)\) is the precise representative of \(K * \mu\) at \(x\) and then \eqref{eqEstimateLipschitz} is satisfied thanks to \Cref{MaximalandSobolev}.

To verify the \(L^{p}\)~estimate of \(I\), we apply the maximal inequality in \(L^{p}(\R^{N})\) and the interpolation argument in the beginning of the proof to get
\[{}
\bignorm{\cM\abs{\nabla (K * \mu)}}_{L^{p}(\R^{N})}
\le \C \norm{\nabla(K * \mu)}_{L^{p}(\R^{N})}
\le \C \norm{\mu}_{L^{p}(\R^{N})}.
\]
The estimate for \(I\) thus follows since \(R \in L^{1}(\R^{N})\) and then, by Young's inequality,
\[{}
\norm{R * \abs{\mu}}_{L^{p}(\R^{N})}
\le \norm{R}_{L^{1}(\R^{N})}  \norm{\mu}_{L^{p}(\R^{N})}.
\qedhere
\]
\end{proof}

\section{Proof of \Cref{theoremDifferentiability}}
\label{sectionDifferentiability}

\resetconstant
When \(p > 1\), it is convenient to equip the space \(L^{p}\lw(\R^{N})\) of weak-\(L^{p}\) functions with the norm
\[{}
\norm{f}_{L^{p}\lw(\R^{N})}
\vcentcolon= \sup{\biggl\{ \frac{1}{\abs{A}^{\frac{p-1}{p}}} \int_{A}{\abs{f}} : A \subset \R^{N}\ \text{has finite measure} \biggr\}},
\]
which is equivalent to the quasinorm \(\seminorm{f}_{L^{p}\lw(\R^{N})}\); see e.g.~the proof of Proposition~5.6 in \cite{Ponce2016}.
A straightforward application of Fubini's theorem gives the following counterpart of Young's inequality in weak-\(L^{p}\) spaces for \(p > 1\):
\begin{equation}
\label{eqYoung}
\norm{K * \mu}_{L^{p}\lw(\R^{N})}
\le \norm{K}_{L^{p}\lw(\R^{N})} \norm{\mu}_{\cM(\R^{N})},
\end{equation}
where \(\mu\) is any finite Borel measure in \(\R^{N}\); see \cite{BenilanBrezisCrandall}*{Lemma~A.4}.

\begin{proof}[Proof of \Cref{theoremDifferentiability}]
Given \(t > 0\), we rely on the Calderón-Zygmund decomposition~\eqref{eqDecompositionCZThm}, namely
\begin{equation*}
\mu{}
= g \dif x + \sum_{n = 0}^{\infty}{b_{n}},
\end{equation*}
where \(g \in (L^{1} \cap L^{\infty})(\R^{N})\) and each measure \(b_{n}\) satisfies \(b_{n}(Q_{n}) = 0\), \(\abs{b_{n}}(\R^{N} \setminus Q_{n}) = 0\) and
\begin{equation}
	\label{eqEstimateBadCubes}
	\norm{b_{n}}_{\cM(\R^{N})}
	\le 2 \abs{\mu}(Q_{n}).
\end{equation}
The half-closed cubes \(Q_{n}\) are disjoint and given by the Whitney covering of the open set 
\[{}
\R^{N} \setminus F
= \{\cM\mu > t\}.
\] 
Moreover,
\[{}
	K*\mu(x){}
	= K*g(x) + \sum_{n = 0}^{\infty}{K*b_{n}(x)}
	\quad \text{for every \(x \in \R^{N} \setminus S\),}
\]
where \(S \subset \R^{N}\) is the Lebesgue-negligible set given by \eqref{eqDecompositionSetS}.
In particular, this identity holds almost everywhere in \(\R^{N}\).

At a point \(y \in \R^{N} \setminus S\) where \(K * g\) and all \(K * b_{n}\) are approximately differentiable, we have
\begin{equation}
\label{eqTaylorSeries}
\abs{K * \mu - P_{y}}
\le \abs{K * g - T_{y}^{1}(K * g)}
+ \sum_{n = 0}^{\infty}{\abs{K * b_{n} - T_{y}^{1}(K * b_{n})}}
\end{equation}
almost everywhere in \(\R^{N}\), where \(P_{y}\) is the affine function
\[{}
P_{y} \vcentcolon= T_{y}^{1}(K * g) + \sum_{n = 0}^{\infty}{T_{y}^{1}(K * b_{n})},
\]
provided that the series in the right-hand side converges.

From \Cref{remarkLp}, since \(g \in (L^{1} \cap L^{\infty})(\R^{N})\) the function \(K * g\) has a distributional gradient in \(L^{p}(\R^{N}; \R^{N})\) and then is \(L^{p}\)-differentiable almost everywhere in \(\R^{N}\) for every \(1 < p < \infty\).{}
In particular, for \(p = \frac{N}{N-1}\),
\begin{equation}
\label{eqDiffLpGood}
\lim_{r \to 0}{\frac{\norm{K * g - T_{y}^{1}(K * g)}_{L^{\frac{N}{N-1}}(B_{r}(y))}}{r^{N}} } = 0
\quad \text{for almost every \(y \in \R^{N}\).}
\end{equation}
Observe that, by smoothness of the functions \(K * b_{n}\) in a neighborhood of \(y \in F\), each term in the series in the right-hand side of \eqref{eqTaylorSeries} evaluated at a point \(x\) behaves like \(o(\abs{x - y})\) as \(x \to y\).{}
Thus, for every \(n \in \N\) we also have
\begin{equation}
\label{eqDiffLpBad}
\lim_{r \to 0}{\frac{\norm{K * b_{n} - T_{y}^{1}(K * b_{n})}_{L^{\frac{N}{N-1}}(B_{r}(y))}}{r^{N}} } = 0
\quad \text{for every \(y \in F\).}
\end{equation}

To handle the fact that we are dealing with infinitely many terms in \eqref{eqTaylorSeries}, we need a uniform estimate of the tail of the series.
To this end, we take \(y \in F \).{}
For every \(n \in \N\), we have \(y \in \R^{N} \setminus 2Q_{n}\) and then, by \Cref{lemmaDecayEstimateClose} applied at \(y\),
\[
\abs{\nabla(K * b_{n})(y)}
\le \frac{C''' \ell_{n}}{\abs{y - \bar{z}_{n}}^{N + 1}} \norm{b_{n}}_{\cM(\R^{N})},
\]
where \(\ell_{n}\) is the side-length and \(\bar{z}_{n}\) is the center of \(Q_{n}\).
By \eqref{eqComparisonMarcinkiewicz} in \Cref{remarkMarcinkiewicz}, we also have
\begin{equation}
	\label{eqMarcinkiewicz}
	\frac{\ell_{n}}{\abs{y - \bar{z}_{n}}^{N + 1}} \norm{b_{n}}_{\cM(\R^{N})}
	\le \C \int_{Q_{n}}{\frac{d(z, F)}{\abs{y - z}^{N + 1}} \dif\abs{\mu}(z)}.
\end{equation}
Hence,
\begin{equation}
\label{eqGradientAnyCube}
\abs{\nabla(K * b_{n})(y)}
\le \C \int_{Q_{n}}{\frac{d(z, F)}{\abs{y - z}^{N + 1}} \dif\abs{\mu}(z)}.
\end{equation}

Next, let \(J \in \N\) and \(r > 0\), and let \(\epsilon > 0\) be given by \Cref{lemmaDecayEstimateClose}.{}
Observe that \(\epsilon\) only depends on the dimension \(N\).
We divide the cubes \(Q_{n}\) with indices \(n \ge J\) in two disjoint classes, according to their distances from the point \(y\) as follows:
 \(\mathcal{F}_{J, r}\) is the subset of indices \(n \ge J\) such that \(\epsilon \abs{y - \bar{z}_{n}} > r\) and \(\mathcal{C}_{J, r}\) is the subset of indices \(n \ge J\) such that \(\epsilon \abs{y - \bar{z}_{n}} \le r\).{}
 
The class  \(\mathcal{F}_{J, r}\) keeps track of the cubes which are \emph{far from \(y\)}.{}
We claim that, for every \(n \in \mathcal{F}_{J, r}\)\,,
\begin{equation}
	\label{eqCubesFar}
	\norm{K * b_{n} - T_{y}^{1}(K * b_{n})}_{L^{\infty}(B_{r}(y))}
	\le \C \, r \int_{Q_{n}}{\frac{d(z, F)}{\abs{y - z}^{N + 1}} \dif\abs{\mu}(z)}.
\end{equation}
Indeed, for every \(x \in B_{r}(y)\) and \(n \in \mathcal{F}_{J, r}\)\,,
\[{}
\abs{x - y} < r < \epsilon \abs{y - \bar{z}_{n}}.
\]
By \Cref{lemmaDecayEstimateClose} (reversing the roles of \(x\) and \(y\)) and \eqref{eqMarcinkiewicz}, we then have
\begin{equation}
\label{eqOrderZeroAny}
\begin{split}
\abs{K * b_{n}(x) - K * b_{n}(y)}
& \le \frac{C''' \ell_{n}}{\abs{y - \bar{z}_{n}}^{N + 1}} \norm{b_{n}}_{\cM(\R^{N})} \abs{x - y}\\
& \le \C \, r \int_{Q_{n}}{\frac{d(z, F)}{\abs{y - z}^{N + 1}} \dif\abs{\mu}(z)}.
\end{split}
\end{equation}
Combining \eqref{eqOrderZeroAny} and \eqref{eqGradientAnyCube}, we deduce \eqref{eqCubesFar}.
The latter implies that, for every \(n \in \mathcal{F}_{J, r}\)\,,
\begin{equation}
	\label{eqCubesFarWeak}
	\norm{K * b_{n} - T_{y}^{1}(K * b_{n})}_{L^{\frac{N}{N-1}}\lw(B_{r}(y))}
	\le \C \, r^{N} \int_{Q_{n}}{\frac{d(z, F)}{\abs{y - z}^{N + 1}} \dif\abs{\mu}(z)}.
\end{equation}

The class \(\mathcal{C}_{J, r}\) gathers the cubes which are \emph{close to \(y\)}.{}
We claim in this case that, for every \(n \in \mathcal{C}_{J, r}\)\,,
\begin{equation}
	\label{eqCubesClose}
	\norm{K * b_{n} - T_{y}^{1}(K * b_{n})}_{L^{\frac{N}{N-1}}\lw(B_{r}(y))}
	\le \C \biggl(\abs{\mu}(Q_{n}) + r^{N} \int_{Q_{n}}{\frac{d(z, F)}{\abs{y - z}^{N + 1}} \dif\abs{\mu}(z)}\biggr).
\end{equation}
Indeed, the decay assumption~\eqref{eqKernel} on \(K\) implies that \(K \in L^{\frac{N}{N-1}}\lw (\R^{N})\).{}
Then, by Young's inequality \eqref{eqYoung} and \eqref{eqEstimateBadCubes} we get 
\begin{equation}
\label{eqFunctionCubeCloseWeakNorm}
\norm{K * b_{n}}_{L^{\frac{N}{N-1}}\lw(B_{r}(y))}
\le \norm{K}_{L^{\frac{N}{N-1}}\lw(\R^{N})} \norm{b_{n}}_{\cM(\R^{N})}
\le \C \abs{\mu}(Q_{n}).
\end{equation}
Using \Cref{lemmaDecayEstimateFar} and \eqref{eqMarcinkiewicz}, we also have
\[{}
\abs{K * b_{n}(y)}
\le \frac{C'' \ell_{n}}{\abs{y - \bar{z}_{n}}^{N}} \norm{b_{n}}_{\cM(\R^{N})}
\le \Cl{cte-1410} \abs{y - \bar{z}_{n}} \int_{Q_{n}}{\frac{d(z, F)}{\abs{y - z}^{N+1}} \dif\abs{\mu}(z)}.
\]
Since \(\epsilon\abs{y - \bar{z}_{n}} \le r\), we thus have
\begin{equation}
\label{eqFunctionCubeClose}
\abs{K * b_{n}(y)}
\le \frac{\Cr{cte-1410} r}{\epsilon} \int_{Q_{n}}{\frac{d(z, F)}{\abs{y - z}^{N + 1}} \dif\abs{\mu}(z)}.
\end{equation}
Estimate \eqref{eqCubesClose} now follows from the combination of \eqref{eqFunctionCubeCloseWeakNorm}, \eqref{eqFunctionCubeClose} and \eqref{eqGradientAnyCube}.

Note that for every \(n \in \mathcal{C}_{J, r}\) we have \(Q_{n} \subset B_{\gamma r}(y)\), for some constant \(\gamma > 0\) depending on \(N\).{}
Since \(\bigcup\limits_{n \in \mathcal{C}_{J, r}}{Q_{n}} \subset B_{\gamma r}(y) \setminus F\) and the cubes \(Q_{n}\) are disjoint, we deduce from \eqref{eqCubesFarWeak} for  \(n \in \mathcal{F}_{J, r}\) and \eqref{eqCubesClose} for  \(n \in \mathcal{C}_{J, r}\) that
\begin{multline*}
\sum_{n = J}^{\infty}{\norm{K * b_{n} - T_{y}^{1}(K * b_{n})}_{L^{\frac{N}{N-1}}\lw(B_{r}(y))}}\\
\le \Cl{cte-1304} \biggl(\abs{\mu}(B_{\gamma r}(y) \setminus F) + r^{N} \int_{\bigcup\limits_{n = J}^{\infty}{Q_{n}}}{\frac{d(z, F)}{\abs{y - z}^{N + 1}} \dif\abs{\mu}(z)}\biggr).
\end{multline*}
Recall that for almost every \(y \in F\) we have
\begin{equation}
\label{eqConditionsF}
\lim_{r \to 0}{\frac{\abs{\mu}(B_{\gamma r}(y) \setminus F)}{r^{N}}}
= 0
\quad \text{and} \quad
\int_{\R^{N} \setminus F}{\frac{d(z, F)}{\abs{y - z}^{N + 1}} \dif\abs{\mu}(z)} < \infty.
\end{equation}
The first assertion follows from  the Besicovitch differentiation theorem, while the second one is a consequence of
inequality \eqref{eqMarcinkiewiczRemark} in \Cref{remarkMarcinkiewicz} satisfied by the Marcinkiewicz integral.
At a point \(y\) where the first property in \eqref{eqConditionsF} holds, we have
\[{}
\limsup_{r \to 0}{\frac{\sum\limits_{n = J}^{\infty}{\norm{K * b_{n} - T_{y}^{1}(K * b_{n})}_{L^{\frac{N}{N-1}}\lw(B_{r}(y))}}}{r^{N}}}
\le \Cr{cte-1304} \int_{\bigcup\limits_{n = J}^{\infty}{Q_{n}}}{\frac{d(z, F)}{\abs{y - z}^{N + 1}} \dif\abs{\mu}(z)}.
\]  
It thus follows from \eqref{eqTaylorSeries}, \eqref{eqDiffLpGood} and \eqref{eqDiffLpBad} that
\[{}
\limsup_{r \to 0}{\frac{\norm{K * \mu - P_{y}}_{L^{\frac{N}{N-1}}\lw(B_{r}(y))}}{r^{N}}}
\le \Cr{cte-1304}\int_{\bigcup\limits_{n = J}^{\infty}{Q_{n}}}{\frac{d(z, F)}{\abs{y - z}^{N + 1}} \dif\abs{\mu}(z)}.
\]
At a point \(y\) where the second property in \eqref{eqConditionsF} holds, the integral in the right-hand side converges to zero as \(J \to \infty\) and we deduce that the limsup in the left-hand side vanishes.{}
Hence, by uniqueness of the Taylor approximation we have 
\[{}
P_{y} = T_{y}^{1}(K * \mu){}
\]
and \(K * \mu\) is \(L^{\frac{N}{N-1}}\lw\)-differentiable at almost every \(y \in F\).
Since \(F = F_{t}\) satisfies \eqref{eqWeakMaximalInequality}, the conclusion of the theorem then follows by letting \(t \to \infty\).
\end{proof}

\section{Existence of $\apDD{u}$ when $\Delta u$ is a mesure}\label{5}

Let \(E : \R^{N} \setminus \{0\} \to \R\) be the fundamental solution of \(-\Delta\) defined in dimension \(N \ge 3\) by
\[{}
E(x) = \frac{1}{(N-2)\sigma_{N}} \frac{1}{\abs{x}^{N-2}},
\]
where \(\sigma_{N}\) denotes the area of the unit sphere in \(\R^{N}\).{}
Since $E \in (L^1+L^\infty)(\R^{N})$, the Newtonian potential \(E * \mu\) is defined almost everywhere for a finite Borel measure \(\mu\) in \(\R^{N}\) and belongs to \(L^{1}\loc(\R^{N})\).
In addition, \(E * \mu \in W^{1, p}\loc(\R^{N})\) for every \(1\leq p< \frac{N}{N-1} \) and one has
\[{}
- \Delta (E * \mu){}
= \mu{}
\quad \text{in the sense of distributions in \(\R^{N}\);}
\]
see Example~2.12 in \cite{Ponce2016}.
Since 
\(\nabla(E * \mu)= (\nabla E) * \mu\) and
\[{}
\nabla E(x) 
= - \frac{1}{\sigma_{N}} \frac{x}{\abs{x}^{N}}
\]
satisfies \eqref{eqKernel}, we have by \Cref{theoremADSingularIntegral} that \(\nabla(E * \mu)\) is approximately differentiable almost everywhere and its approximate derivative \(\apDD{(E * \mu)} = \apD{((\nabla E) * \mu)}\) satisfies
\begin{equation}
	\label{eqEstimateNewtonian}
	\seminorm{\apDD{(E * \mu)}}_{L^{1}\lw(\R^{N})}
	\le C \norm{\mu}_{\cM(\R^{N})}.
\end{equation}
We apply this estimate to identify the trace of \(\apDD{(E * \mu)}\) in terms of \(\mu\):

\begin{proposition}\label{propositionEqualityLaplacian}
Let \(N \ge 3\).{}
For every finite Borel measure $\mu$ in \(\R^{N}\), the absolutely continuous part of \(\mu\) with respect to the Lebesgue measure satisfies
\[{}
\mu\la = - \Trace{( \apDD(E * \mu))} \dif x.{}
\]
\end{proposition}

	Observe that for \(f \in C_{c}^{\infty}(\R^{N})\), the Newtonian potential \(E * f\) is a smooth function.
	Thus, \(\apDD(E * f)\) is the classical second-order derivative of \(E * f\).{}
	Since \(E * f\) solves the Poisson equation with density \(f\), we get
	\[{}
	f 
	= - \Delta (E * f)
	= - \Trace{( \apDD(E * f))}
	\quad \text{in \(\R^{N}\)},
	\]
	which is \Cref{propositionEqualityLaplacian} for smooth functions.
	
	Next, for a finite Borel measure \(\mu\), the Newtonian potential \(E * \mu\) is smooth and harmonic in \(\R^{N} \setminus \supp{\mu}\).{}
	Thus,
	\[{}
	\Trace{(\apDD(E * \mu))}
	= \Delta(E * \mu){}
	= 0
	\quad \text{in \(\R^{N} \setminus \supp{\mu}\).}
	\]
	In particular, when the support \(\supp{\mu}\) is negligible with respect to the Lebesgue measure, one has
	\[{}
	\Trace{(\apDD(E * \mu))}
	= 0
	\quad \text{almost everywhere in \(\R^{N}\).}
	\]
	The proof of \Cref{propositionEqualityLaplacian} is based on an approximation argument that relies on 
	estimate \eqref{eqEstimateNewtonian} and these two cases.
	
\resetconstant	
\begin{proof}[Proof of \Cref{propositionEqualityLaplacian}]
	We first assume that \(\mu\) is absolutely continuous with respect to the Lebesgue measure.
	Thus, \(\mu = \mu\la = f \dif x\) for some \(f \in L^{1}(\R^{N})\).{}
	Take a sequence \((f_{n})_{n \in \N}\) in \(C_{c}^{\infty}(\R^{N})\) that converges to \(f\) in \(L^{1}(\R^{N})\).{}
	For every \(n \in \N\), 
	\[{}
	f_{n} 
	= - \Trace{( \apDD(E * f_{n}))}
	\quad \text{in \(\R^{N}\).}
	\]
	Thus, by the triangle inequality and linearity of the approximate derivative,
	\[{}
	\bigabs{\Trace{(\apDD{(E * f)})}  + f_{n}}
	\le \bigabs{\apDD(E * f) - \apDD{(E * f_{n})}}
	= \bigabs{\apDD(E * (f - f_{n}))}.
	\]
	Applying estimate \eqref{eqEstimateNewtonian} to \(f - f_{n}\), we thus have
	\[{}
	\begin{split}
	\bigseminorm{\Trace{(\apDD{(E * f)})}  + f_{n} }_{L^{1}\lw(\R^{N})}
	& \le \bigseminorm{\apDD(E * (f - f_{n}))}_{L^{1}\lw(\R^{N})}\\
	& \le C \norm{(f - f_{n}) \dif x}_{\cM(\R^{N})}
	= C \norm{f - f_{n}}_{L^{1}(\R^{N})}.
	\end{split}
	\]
	Hence, the sequence \((f_{n})_{n \in \N}\) converges in measure simultaneously to \(f\) and \(-\Trace{(\apDD{(E * f)})}\).{}
	By uniqueness of the limit, we deduce that 
	\[{}
	- \Trace{( \apDD(E * f))}
	= f
	\quad \text{almost everywhere in \(\R^{N}\),}
	\]
	when \(\mu = f \dif x\) is absolutely continuous with respect to the Lebesgue measure.
	
	We now assume that \(\mu\) is singular with respect to the Lebesgue measure.
	Let \(S \subset \R^{N}\) be a negligible Borel set such that \(\abs{\mu}(\R^{N} \setminus S) = 0\).{}
	By inner regularity of \(\abs{\mu}\), there exists a sequence of compact sets \(K_{n} \subset S\) such that 
	\[{}
	\lim_{n \to \infty}{\abs{\mu}(S \setminus K_{n})} = 0.{}
	\]
	Each measure \(\mu_{n} = \mu\lfloor_{K_{n}}\) is supported in the negligible set \(K_{n}\).{}
	Since \(E * \mu_{n}\) is harmonic in \(\R^{N} \setminus K_{n}\), we thus have
	\[{}
	\Trace{(\apDD(E * \mu_{n}))}
	= 0
	\quad \text{almost everywhere in \(\R^{N}\).}
	\]
	Again, by linearity of the approximate derivative,
	\[{}
	\bigabs{\Trace{(\apDD{(E * \mu)})}}
	\le \bigabs{\apDD(E * \mu) - \apDD{(E * \mu_{n})}}	
	= \bigabs{\apDD(E * (\mu - \mu_{n}))}
	\]
	almost everywhere in \(\R^{N}\).
	Estimate~\eqref{eqEstimateNewtonian} applied to \(\mu - \mu_{n}\) then implies
	\[{}
	\bigseminorm{\Trace{(\apDD{(E * \mu)})}}_{L^{1}\lw(\R^{N})}
	\le C \norm{\mu - \mu_{n}}_{\cM(\R^{N})}
	= \abs{\mu}(\R^{N} \setminus K_{n})
	= \abs{\mu}(S \setminus K_{n}).
	\]	
	As \(n \to \infty\), the right-hand side converges to zero.
	Hence, 
	\[{}
	\Trace{(\apDD{(E * \mu)})} = 0
	\quad \text{almost everywhere in \(\R^{N}\),}
	\]
	when \(\mu\) is singular with respect to the Lebesgue measure.
	
	The proof now follows from the linearity of \(E * \mu\) since any finite Borel measure \(\mu\) has a decomposition of the form
	\[{}
	\mu{}
	= \mu\la + \mu\ls{}
	= f \dif x + \mu\ls
	\quad \text{with \(f \in L^{1}(\R^{N})\),}
	\]
	where \(\mu\ls\) is the singular part of \(\mu\) with respect to the Lebesgue measure.
	By the two cases considered above, we have
	\[{}
	\Trace{(\apDD{(E * \mu)})}
	= \Trace{(\apDD{(E * f)})} + \Trace{(\apDD{(E * \mu\ls)})}
	= - f
	\]
	almost everywhere in \(\R^{N}\).
\end{proof}

In dimension \(N = 2\), the fundamental solution of \(-\Delta\) is
\[{}
E(x) = \frac{1}{2\pi} \log{\frac{1}{\abs{x}}}
\]
and the Newtonian potential \(E * \mu\) is well-defined for every finite Borel measure with \emph{compact support} in \(\R^{2}\).{}
The counterpart of \Cref{propositionEqualityLaplacian} holds for these measures, with the same proof.

In every dimension \(N \ge 2\), \Cref{propositionEqualityLaplacian} has a counterpart for solutions of the Dirichlet problem
\begin{equation}
	\label{eqDirichletProblem}
\left\{
\begin{alignedat}{2}
	- \Delta u & = \mu && \quad \text{in \(\Omega\),}\\
	u & = 0 && \quad \text{in \(\partial\Omega\),}
\end{alignedat}
\right.
\end{equation}
involving a finite Borel measure \(\mu\) in a smooth bounded open subset \(\Omega \subset \R^{N}\).
By a solution of \eqref{eqDirichletProblem}, we mean that \(u \in W_{0}^{1, 1}(\Omega)\) satisfies	
	\[{}
	- \Delta u = \mu{}
	\quad \text{in the sense of distributions in \(\Omega\).}
	\]
Littman, Stampacchia and Weinberger~\cite{LittmanStampacchiaWeinberger1963} proved that the Dirichlet problem above has a unique solution for every \(\mu\).{}
This solution has additional imbedding properties that can be formulated in terms of weak-Lebesgue spaces.
For example,  using Stampacchia's truncation method one shows in dimension \(N \ge 3\) that
\[{}
\seminorm{u}_{L^{\frac{N}{N-2}}\lw(\Omega)} + \seminorm{\nabla u}_{L^{\frac{N}{N-1}}\lw(\Omega)}
\le C \norm{\Delta u}_{\cM(\Omega)},
\]
where \(C > 0\) depends on \(N\); see \cite{Ponce2016}*{Proposition~5.7}.
A second-order counterpart of this inequality is

\begin{proposition}
	\label{propositionDirichletProblem}
	Let \(N \ge 2\) and let \(\Omega \subset \R^{N}\) be a smooth bounded open set.
	For every finite Borel measure \(\mu\) in \(\Omega\), the solution \(u\) of the Dirichlet problem \eqref{eqDirichletProblem} has a second-order approximate derivative \(\apDD{u}\) almost everywhere in \(\Omega\) that satisfies
	\[{}
	(\Delta u)\la = \Trace{( \apDD u)} \dif x
	\quad \text{and} \quad
	\seminorm{\apDD{u}}_{L^{1}\lw(\Omega)}
	\le C' \norm{\Delta u}_{\cM(\Omega)},
	\]
	for some constant \(C' > 0\) depending on \(N\).
\end{proposition}

\resetconstant
\begin{proof}
	Let \(U : \R^{N} \to \R\) be the function defined by
	\[{}
	U(x)= 
	\begin{cases}
	u 	& \text{if \(x \in \Omega\),}\\
	0 	& \text{if \(x \in \R^{N} \setminus \Omega\).}
	\end{cases}
	\]
	From Poincaré's balayage method, see \cite{Ponce2016}*{Corollary~7.4}, one has that \(\Delta U\) is a finite Borel measure with compact support in \(\R^{N}\) and
	\[{}
	\norm{\Delta U}_{\cM(\R^{N})}
	\le 2 \norm{\Delta u}_{\cM(\Omega)}.
	\]
	Since \(U\) has compact support in \(\R^{N}\), we have 
	\[{}
	U = E * (-\Delta U)
	\quad \text{in \(\R^{N}\).}
	\]
	This identity is indeed true for functions in \(C_{c}^{\infty}(\R^{N})\) and the general case follows by approximation using \(\rho_{n} * U\), where \((\rho_{n})_{n \in \N}\) is a sequence of mollifiers in \(C_{c}^{\infty}(\R^{N})\).
	We deduce from \Cref{propositionEqualityLaplacian} (and its counterpart in dimension \(2\)) that \(u = U\) has a second-order approximate derivative almost everywhere in \(\Omega\) with
	\[{}
	(\Delta u)\la = (\Delta U)\la = \Trace{( \apDD U)} \dif x = \Trace{( \apDD u)} \dif x
	\]
	 and
	\[{}
	\seminorm{\apDD{u}}_{L^{1}\lw(\Omega)}
	\le \Cl{cte-1229} \norm{\Delta U}_{\cM(\R^{N})}
	\le 2\Cr{cte-1229} \norm{\Delta u}_{\cM(\Omega)}.
	\qedhere
	\]	
\end{proof}

\section{Proof of \Cref{th:main1}}
\label{sectionProofThm}

Given a bounded open subset $\omega \Subset \Omega$, let $\varphi \in C^\infty_c(\Omega)$ be such that $\varphi = 1$ on $\omega$. 
The measure  \(\mu = -\Delta(u\varphi)\) is finite, has compact support in \(\Omega\) and can be written as 
\[{}
\mu 
= - \bigl( \Delta u \, \varphi + 2 \nabla u \cdot \nabla \varphi +u \Delta \varphi \bigr){}
\quad \text{in the sense of distributions in \(\R^{N}\);}
\]
see \cite{Ponce2016}*{Proposition~6.11}. 
In particular, 
\begin{equation}
	\label{eqMuSubset}
	\mu = -\Delta u
	\quad \text{in \(\omega\).}
\end{equation}

We extend the measure \(\mu\) to \(\R^{N}\) as zero on subsets of \(\R^{N} \setminus \Omega\).
The Newtonian potential \(E * \mu\) and \(u\varphi\) satisfy the same Poisson equation in \(\Omega\).{}
Thus, by Weyl's lemma we have
\[{}
E * \mu{}
= u \varphi + h
\quad \text{almost everywhere in \(\Omega\),}
\]
where $h$ is a harmonic function in $\Omega$. 
By \Cref{propositionEqualityLaplacian} (and its counterpart in dimension \(2\)), we deduce that $\nabla(u \varphi)$ is approximately differentiable almost everywhere in $\Omega$ and
\begin{equation}
	\label{eqEqualityLaplacianLocal}
	\mu\la{}
= - \Trace{(\apDD (u \varphi))} \dif x.
\end{equation}
Since $\nabla u = \nabla(u \varphi)$ in $\omega$ and the notion of approximate derivative is local, \(\nabla u\) is approximately differentiable almost everywhere in \(\omega\) and
\(\apDD u
= \apDD (u \varphi)\) in \(\omega\).
It then follows from \eqref{eqMuSubset} and \eqref{eqEqualityLaplacianLocal} that
\[
(\Delta u)\la
= - \mu\la{}
= \Trace{(\apDD u)} \dif x
\quad \text{in \(\omega\)}.
\]
Since this is true for every bounded open subset $\omega \Subset \Omega$, we have \((\Delta u)\la
= \Trace{(\apDD u)} \dif x\) in \(\Omega\).

To conclude the proof we rely on \eqref{eqDensityPoint_ter}, that implies
\[{}
\apDD{u}
= \apD{(\nabla u)}
= 0
\quad \text{almost everywhere on $\{\nabla u=e \}$}
\]
for every \(e \in \R^{N}\).
Thus,
\begin{equation}
\label{eqProofLevelSets}
(\Delta u)\la{}
= \Trace{(\apDD u)} \dif x
=0 \quad \text{almost everywhere on $\{\nabla u=e \}$}.
\end{equation}
Since $u \in W^{1,1}\loc(\Omega)$, we also have $\nabla u=0$ almost everywhere on every level set $\{u=\alpha \}$ with \(\alpha \in \R\).
Hence, applying \eqref{eqProofLevelSets} with \(e = 0\) we also get 
\[{}
(\Delta u)\la=0{}
\quad \text{almost everywhere on $\{u=\alpha \}$.}
\] 
This completes the proof of the theorem.
\qed

\section{Two applications}\label{7}

\subsection{Level sets of subharmonic functions}
As a first application of \Cref{th:main1} we give a new proof of Theorem~B.1 of Frank and Lieb~\cite{FrankLieb} about level sets of subharmonic functions.

\begin{proposition}\label{FrankLieb}
Let \(\Omega \subset \R^{N}\) be an open set and let $u\in L^1\loc(\Omega)$ be such that $\Delta u$ is a locally finite Borel measure in \(\Omega\).
If
\begin{equation}\label{new_hypo}
|\Delta u|\geq\theta\dif x,
\end{equation}
where \(\theta\) is a Borel-measurable function in \(\Omega\) such that
\begin{equation}\label{new_hypo-bis}
\theta>0
\quad \text{almost everywhere in $\Omega$,}
\end{equation} 
then the level set $\{u=\alpha\}$ is Lebesgue-negligible for every $\alpha \in \R$. 
\end{proposition}

\begin{proof}
Let $\Delta u= (\Delta u)\la +(\Delta u)\ls$ be the decomposition of $\Delta u$ in terms of an absolutely continuous and a singular part with respect to the Lebesgue measure, and let $S \subset \Omega$ be a Borel-measurable set such that $|S|=0$ and 
$\abs{(\Delta u)\ls} (\Omega \setminus S)=0$.
For every \(\alpha \in \R\), we have
\begin{equation*}
\{u = \alpha\}
\subset \bigl(\{u = \alpha\} \setminus S\bigr) \cup S.
\end{equation*}
Since \(S\) is negligible, it suffices to prove that \(\{u = \alpha\} \setminus S\) is negligible.
To this end, we write the estimate
\[{}
\int_{\{u=\alpha\} \setminus S} \theta\dif x\leq \int_{\{u=\alpha\} \setminus S}{|\Delta u|}
\le \int_{\{u=\alpha\}}{\abs{(\Delta u)\la}} + \int_{\Omega \setminus S}{\abs{(\Delta u)\ls}}.
\]
Both integrals in the right-hand side vanish: The first one because of \Cref{th:main1} and the second one by the choice of \(S\).
Hence, the integral in the left-hand side also vanishes, and then, by assumption on \(\theta\), the set \(\{u=\alpha\} \setminus S\) 
must be negligible.
\end{proof}

Assumptions \eqref{new_hypo} and \eqref{new_hypo-bis} are implied by
\begin{equation}\label{old_hypo}
\int_K\Delta u\neq 0
\quad\text{for any compact set $K \subset \Omega$ with \(|K|>0\),}
\end{equation} 
which is weaker than the assumption made in \cite{FrankLieb}, namely positivity of \(\int_K\Delta u\). 
Indeed, let $S \subset \Omega$ be a Lebesgue-negligible set where the singular part of $|\Delta u|$ with respect to the Lebesgue measure $\mathrm{d}x$ is concentrated.
Condition \eqref{old_hypo} implies, by the Hahn decomposition and inner approximation of \(\Delta u\), 
\[
\int_B |\Delta u|>0
\quad \text{for any Borel set $B \subset \Omega$ with \(|B|>0\).}
\] 
We now take $\theta$ as the density of the absolutely continuous part of $|\Delta u|$ with respect to $\mathrm{d}x$.
Since the integral above vanishes with $B=\{\theta=0\}\setminus S$, we see that $\{\theta=0\}\setminus S$ and then
$\{\theta=0\}$ must be Lebesgue-negligible, which gives \eqref{new_hypo} and \eqref{new_hypo-bis}. 

As a consequence of \Cref{FrankLieb} above, we deduce Proposition~A.1 from \cite{FrankLieb}:

\begin{corollary}
\label{FrankLieb2}
Let $\Omega \subset \R^N$ be an open set and let $u : \Omega \to \R$ be a continuous function.
If there exists \(\epsilon > 0\) such that
\begin{equation}
	\label{eqViscosity}
\Delta u \ge \epsilon{}
\quad \text{in the viscosity sense in \(\Omega\)},
\end{equation}
then $\{u=\alpha\}$ is Lebesgue-negligible for every $\alpha \in \R$. 
\end{corollary}

\begin{proof}
We recall that a continuous function $u$ satisfies \eqref{eqViscosity} whenever one has 
\(\Delta \varphi(x) \geq \epsilon\) for every $x \in \Omega$ and every $\varphi \in C^2(\Omega)$ such that $\varphi \ge u$ in a neighborhood of $x$ and \(\varphi(x) = u(x)\).
By the relation between viscosity and distributional solutions, see \cite{Ishii1995}*{Theorem~1} and also \cites{Lions1983,JuutinenLindqvistManfredi2001}, we have
\[{}
\Delta u \ge \epsilon{}
\quad
\text{in the sense of distributions in \(\Omega\).}
\] 
This implies that \(\Delta u \) is a locally finite Borel measure in $\Omega$ with \(\Delta u\geq\epsilon \dif x\),
so that \Cref{FrankLieb} with $\theta=\epsilon$ gives the result.
\end{proof}

As Frank and Lieb point out in their paper, the conclusion of \Cref{FrankLieb2} is false under the assumption $\Delta u \geq 0$, without strict inequality, as shown by the example $u(x)=\max{\{x_1,0\}}$ in \(\R^{N}\). 

\subsection{Limiting vorticities of the Ginzburg-Landau system}
The Ginzburg-Laudau model can be used to describe the phenomenon of 
superconductivity in some materials at low temperature subject to a magnetic field; see \cite{SandierSerfaty}.
The state of a superconducting sample in a domain \(\Omega \subset \R^{2}\) is then described by an order parameter $u: \Omega \rightarrow \mathbb{C}$ and a magnetic potential $A : \Omega \rightarrow \R$ that are local minimizers or merely critical points of the energy functional
\begin{equation}\label{eq:GLenergy}
G_\epsilon(u, A) = \frac{1}{2} \int_{\Omega}{|\nabla^{A} u|^2} + \int_{\Omega}{(h-h_{\text{ex}})^2} + \frac{1}{4\epsilon^2} \int_{\Omega}{(1-|u|^2)^2}.
\end{equation}
Here, $\epsilon > 0$ is a small parameter, the constant $h_{\text{ext}} > 0$ is the intensity of the applied magnetic field, 
\[{}
h \vcentcolon= \curl{A} = -\Div{A^{\perp}}
\]
is the induced magnetic field and 
\[{}
\nabla^{A} u \vcentcolon= \nabla u - i A u
\]
is the covariant gradient.
Regions where the order parameter \(u\) satisfies \(|u| \approx 1\) are in superconductor phase, while the material behaves as a normal conductor in places where \(|u| \approx 0\).
The Euler-Lagrange equation satisfied by a critical point \((u_{\epsilon}, A_{\epsilon})\) of  \eqref{eq:GLenergy} gives (see Proposition~3.6 in \cite{SandierSerfaty})
\begin{equation}\label{eq:2ndGLequation}
-\nabla^\perp h_{\epsilon} = (i u_{\epsilon} | \nabla^{A} u_{\epsilon})
\quad 
\text{in the sense of distributions in \(\Omega\),}
\end{equation}
where 
\[{}
\nabla^\perp h_{\epsilon}
\vcentcolon = \Bigl(-\frac{\partial h_{\epsilon}}{\partial x_{2}}, \frac{\partial h_{\epsilon}}{\partial x_{1}} \Bigr)
\]
and 
\[{}
(i u_{\epsilon} | \nabla^{A_{\epsilon}} u_{\epsilon})
\vcentcolon = \frac{i u_{\epsilon} \, \overline{\nabla^{A_{\epsilon}} u_{\epsilon}} + \overline{i u_{\epsilon}} \, \nabla^{A_{\epsilon}}u_{\epsilon}}{2} \in \R^{2}
\]
is the superconductivity current.
By computing the curl on both sides of \eqref{eq:2ndGLequation}, one obtains
\begin{equation}\label{eq:London}
-\Delta h_\epsilon + h_\epsilon = \mu_\epsilon{}
\quad 
\text{in the sense of distributions in \(\Omega\),}
\end{equation}
where the intrinsic vorticity \(\mu_{\epsilon}\) associated to \((u_{\epsilon}, A_{\epsilon})\) is given by
\[{}
\mu_{\epsilon} \vcentcolon= \curl{(iu_\epsilon | \nabla^{A_{\epsilon}} u_\epsilon)} + \curl{A_{\epsilon}}.
\]
This quantity is an analogue of the distributional Jacobian that is invariant under the gauge transformation
\[{}
(u, A) \longmapsto (u \e^{i f}, A + \nabla f), 
\] 
for any smooth function \(f\).

Under certain conditions in the regime where \(\epsilon \to 0\), a suitable renormalization of \(h_{\epsilon}\) and \(\mu_{\epsilon}\) yields a nontrivial solution of the equation
\begin{equation}
	\label{eqVorticity}
-\Delta h +h=\mu{}
\quad 
\text{in the sense of distributions in \(\Omega\)}
\end{equation}
where $h \in W\loc^{1, 2}(\Omega)$ is called the limiting induced magnetic field and the locally finite Borel measure \(\mu\) is the limiting vorticity.
The limiting magnetic field \(h\) satisfies in addition
\begin{equation}
	\label{eqVorticityBis}
\Div{T_h}=0
\quad 
\text{in the sense of distributions in \(\Omega\),}
\end{equation}
where \(\Div{T_h} = (\Div{T_{h, 1}}, \Div{T_{h, 2}})\) and \(T_{h, i} : \Omega \to \R^{2}\) is defined by
 \begin{equation*}
 T_{h, i, j}
 \vcentcolon= -\frac{\partial h}{\partial x_i}\frac{\partial h}{\partial x_j} + \frac{1}{2}(|\nabla h|^2+h^2)\delta_{ij} \quad \text{for } i, j = 1, 2.
 \end{equation*}

The equation \(\Div{T_h} = 0\) means that $h$ is a critical point of the functional 
\[{}
v \in W^{1, 2}(\Omega) \longmapsto \int_{\Omega}{(|\nabla v|^2 + v^2 )}
\]
with respect to {inner variations of the domain}, i.e.\@ variations of the form $v_t(x) = v(x+tX(x))$ around \(t = 0\) for every vector field $X \in C^\infty_c(\Omega; \R^2)$. 
When \(h\) is a smooth function in \(\Omega\), one finds that
\begin{equation}
	\label{eqDivergenceSmooth}
\Div{T_h} = 
(-\Delta h + h) \nabla h.
\end{equation}
For $h$ that merely belongs to $W^{1, 2}\loc(\Omega)$, \(T_{h}\) belongs to \(L^{1}\loc(\Omega; \R^{2} \times \R^{2})\).{}
In this case, $\Div{T_h}$ is well-defined in the sense of distributions, but the distributional meaning of the right-hand side in \eqref{eqDivergenceSmooth} becomes unclear.

As an application of Theorem~\ref{th:main1}, we identify the absolutely continuous part of the limiting vorticity \(\mu\), in connection with \eqref{eqDivergenceSmooth}, as follows:

\begin{proposition}
	\label{propositionVorticity}
	Let \(\Omega \subset \R^{2}\) be an open set and let \(\mu\) be a locally finite Borel measure in \(\Omega\).
	If \(h \in W\loc^{1, 2}(\Omega)\) satisfies \eqref{eqVorticity} and \eqref{eqVorticityBis},
	then its precise representative \(\Lebesgue{h}\) is locally Lipschitz-continuous in \(\Omega\) and 
	\[{}
	\mu\la = h\chi_{\{ \nabla h=0\}} \dif x.
	\]
\end{proposition}

When \(\mu \in L\loc^{1}(\Omega)\), one thus has
\begin{equation}
\label{eqMuL1}
\mu = h\chi_{\{ \nabla h=0\}}
\quad \text{almost everywhere in \(\Omega\).}
\end{equation}
Such a representation of \(\mu\) was only known for \(\mu \in L\loc^{p}(\Omega)\) with \(p > 1\); see \cite{SandierSerfaty}*{Theorem~13.1}.
As a consequence of \eqref{eqMuL1}, \(h\) satisfies the Schrödinger equation
\[{}
-\Delta h + Vh = 0
\quad \text{in the sense of distributions in \(\Omega\),}
\]
where the potential \(V \vcentcolon= \chi_{\{\nabla h \ne 0\}}\) takes its values in \(\{0, 1\}\). 
A description of the singular part of the limiting vorticity when $\mu$ is merely a measure has been investigated in the paper~\cite{Rodiac2018} by the third author.

The Lipschitz continuity of \(h\) was established by Sandier and Serfaty in \cite[pp.~278--279]{SandierSerfaty}, including the case where \(\mu\) is a measure.
We focus on the new property concerning the identification of \(\mu\la\)\,, but we also present a sketch of their regularity result that is needed in our proof. 

\resetconstant
\begin{proof}[Proof of \Cref{propositionVorticity}]
Using the vector field
\begin{equation}
\label{eqVectorXh}
X_{h} 
\vcentcolon= \frac{1}{2} \biggl( \Bigl(\frac{\partial h}{\partial x_{2}} \Bigr)^{2} - \Bigl(\frac{\partial h}{\partial x_{1}} \Bigr)^{2}, -2 \frac{\partial h}{\partial x_{1}}\frac{\partial h}{\partial x_{2}}  \biggr),
\end{equation}
one can write the components \(T_{h, 1}\) and \(T_{h, 2}\) of \(T\) as
\begin{equation}
	\label{eqComponentsT}
T_{h, 1}
= X_{h} + \frac{1}{2} (h^{2}, 0)
\quad \text{and} \quad{}
T_{h, 2}
= -X_{h}^{\perp} + \frac{1}{2} (0, h^{2}).
\end{equation}
In addition, 
\begin{equation}
	\label{eqNormX}
\abs{\nabla h}^{4} = 4\abs{X_{h}}^{2}.
\end{equation}
By \eqref{eqComponentsT} and the assumption on \(\Div{T_{h}}\), we have that \(X_{h}\) satisfies the div-curl system
\begin{equation*}
\left\{
\begin{aligned}
	\Div{X_{h}} & = - h \, \frac{\partial h}{\partial x_{1}},\\
	\Curl{X_{h}} & =  - h \, \frac{\partial h}{\partial x_{2}},
\end{aligned}
\right.
\end{equation*}
where \(\Curl{X_{h}} = - \Div{X_{h}^{\perp}}\).
Since $h \in W^{1,2}_{\text{loc}}(\Omega)\), one can apply elliptic $L^p$~estimates and a bootstrap argument based on \eqref{eqNormX} to deduce that \(X_{h} \in W^{1, p}\loc(\Omega; \R^{2})\) for every \(1 < p < \infty\).{}
Hence, \(\abs{\nabla h} \in L^{\infty}\loc(\Omega)\) and then $\Lebesgue{h}$ is locally Lipschitz-continuous in \(\Omega\).

We now identify \(\mu\la\) in terms of \(h\).
To this end, take open subsets \(\omega \Subset O \Subset \Omega\) and a sequence of mollifiers  \((\rho_{n})_{n \in \N}\) in \(C_{c}^{\infty}(\R^{2})\) such that \(\omega - \supp{\rho_{n}} \subset O\) for every \(n \in \N\).{}
Denoting $h_n = \rho_n * h$ and \(\mu_{n} = \rho_{n} * \mu\), by linearity of the equation in \eqref{eqVorticity} we then have
\[{}
-\Delta h_{n} + h_{n} = \mu_{n}
\quad
\text{in \(\omega\).}
\]
In this case, \eqref{eqDivergenceSmooth} can be written as
\[{}
\Div{(T_{h_{n}})} = \mu_{n} \nabla h_{n}
\quad
\text{in \(\omega\).}
\]
Since \((h_{n})_{n \in \N}\) converges to \(h\) in \(W^{1, 2}(\omega)\), the sequence \((T_{h_{n}})_{n \in \N}\) converges to \(T_{h}\) in \(L^{1}(\omega; \R^{2} \times \R^{2})\).{}
We thus have
\begin{equation}
	\label{eqLimitDistributions}
\mu_{n} \nabla h_{n} = \Div T_{h_{n}} \overset{*}{\rightharpoonup} \Div T_h = 0
\quad \text{in the sense of distributions in \(\omega\).}
\end{equation}

To give an alternative identification the limit of the sequence \((\mu_{n} \nabla h_{n})_{n \in \N}\) in terms of \(h\) and \(\mu\), we write
\begin{equation*}
\mu_{n} = \rho_{n} * (\mu\la) + \rho_{n} * (\mu\ls).
\end{equation*}
The sequence \((\rho_{n} * (\mu\la))_{n \in \N}\) converges to \(f\) in \(L^{1}(\omega)\), where \(f\) is the density of \(\mu\la\) with respect to the Lebesgue measure, i.e.~\(
\mu\la = f \dif x
\). 
From the Lipschitz-regularity part of the proof, the sequence \((\nabla h_{n})_{n \in \N}\) is uniformly bounded in \(\omega\) and converges to \(\nabla h\) in \(L^{1}(\omega)\).{}
Passing to a subsequence if necessary, we may assume that \((\nabla h_{n})_{n \in \N}\) converges almost everywhere to \(\nabla h\) in \(\omega\). We then write 
\begin{equation}
\rho_n * (\mu\la) \nabla h_n= (\rho_n *(\mu\la)-f)\nabla h_n +f\nabla h_n.
\end{equation}
Since \((\nabla h_n)_{n \in \N}\) is uniformly bounded, the Dominated convergence theorem thus implies that
\begin{equation*}
\rho_{n} * (\mu\la) \, \nabla h_{n} \to f\, \nabla h 
\quad \text{in \(L^{1}(\omega)\).}
\end{equation*}
By uniform boundedness of the sequence \((\nabla h_{n})_{n \in \N}\) in \(\omega\),
\begin{equation}
\label{eqMeasureSingular}
\abs{\rho_{n} * (\mu\ls) \, \nabla h_{n}}
\le \Cl{cte-1924}  \,\rho_{n} * \abs{\mu\ls}.
\end{equation}
Passing to a subsequence if necessary, we may assume that \((\rho_{n} * (\mu\ls) \, \nabla h_{n})_{n \in \N}\) converges weak\(^{*}\) in \((C_{0}(\overline{\omega}))'\) to a finite measure \(\gamma\).{}
We thus have
\[{}
\mu_{n} \nabla h_{n} \overset{*}{\rightharpoonup} f \, \nabla h \dif x + \gamma
\quad \text{weakly$^*$ as measures in \(\omega\).}
\]

By uniqueness of the limit in \eqref{eqLimitDistributions}, we conclude that
\begin{equation}
	\label{eqZero}
	f \, \nabla h \dif x + \gamma = 0
	\quad \text{in the sense of distributions in \(\omega\)}.
\end{equation}
This identity also holds in the sense of measures in \(\omega\); see for instance
 \cite{Ponce2016}*{Proposition~6.12}. 
 By \eqref{eqMeasureSingular}, \(\gamma\) satisfies
\[{}
\abs{\gamma}
\le \Cr{cte-1924} \abs{\mu\ls}
\quad \text{in \(\omega\),}
\]
and in particular is singular with respect to the Lebesgue measure.
We deduce from \eqref{eqZero} that \(f \, \nabla h \dif x = 0 \) and \(\gamma = 0\).{}
Thus,
\[{}
f = 0
\quad \text{almost everywhere in \(\{\nabla h \neq 0\} \cap \omega\).}
\]
On the other hand, by \Cref{th:main1} and \eqref{eqVorticity},{}
\[{}
f = h
\quad \text{almost everywhere in \(\{\nabla h = 0\}\).}
\]
Hence, \(f = h \chi_{\{\nabla h = 0\}}\) almost everywhere in \(\omega\).{}
Since \(\omega \Subset \Omega\) is an arbitrary open subset, the conclusion follows.
\end{proof}

We deduce from \Cref{propositionVorticity} that a limiting vorticity \(\mu \in L^{1}\loc(\Omega)\) is fully described by at most countably many open sets that confine clouds of vortices where the limiting induced magnetic field \(h\) is constant:

\begin{corollary}
\label{corollaryVorticity}
	Let \(\Omega \subset \R^{2}\) be an open set.
	If \(h \in W\loc^{1, 2}(\Omega)\) satisfies \eqref{eqVorticity} and \eqref{eqVorticityBis} with \(\mu \in L^{1}\loc(\Omega)\), then
	\[{}
	\mu = \sum_{j \in J}{m_{j} \chi_{U_{j}}}
	\quad \text{almost everywhere in \(\Omega\),}
	\]
	where \((m_{j})_{j \in \N}\) is the collection of values in \(\R \setminus \{0\}\) such that the level sets \(\{\Lebesgue{h} = m_{j}\}\) have a non-empty interior \(U_{j}\) for every \(j \in J\).
\end{corollary}

Before proving the corollary, we first recall the meaning used by Caffarelli and Salazar in \cite{CaffarelliSalazar:2002} of a viscosity solution \(u: \Omega \to \R\) of the equation
\begin{equation}
	\label{eqViscositySchrodinger}
\Delta u = u 
\quad \text{in \(\{\nabla u \ne 0\}\).}
\end{equation}

\begin{definition}\label{def:viscosity_sense_improved}
A continuous function \(u : \Omega \to \R\) satisfies \eqref{eqViscositySchrodinger} in the viscosity sense whenever both properties hold:
\begin{enumerate}[(i)]
	\item{}
	\label{item-2479}
	for every \(x \in \Omega\) and every polynomial \(P\) of degree at most \(2\) such that \(P \ge u\) in a neighborhood of \(x\) with \(P(x) = u(x)\) and \(\nabla P(x) \ne 0\), 
	\[{}
	\Delta P(x) \ge u(x),{}
	\]
	\item{}
	\label{item-2481} 
	for every \(x \in \Omega\) and every polynomial \(P\) of degree at most \(2\) such that \(P \le u\) in a neighborhood of \(x\) with \(P(x) = u(x)\) and \(\nabla P(x) \ne 0\), 
	\[{}
	\Delta P(x) \le u(x).
	\]
\end{enumerate}
\end{definition}

We then apply the regularity theory developed in \cite{CaffarelliSalazar:2002} to prove the decomposition of the limiting vorticity \(\mu\).{}
To this end, we need the following lemma that clarifies the connection between \eqref{eqViscositySchrodinger} and the equation satisfied by the limiting induced magnetic field.

\begin{lemma}\label{prop:viscosity}
If \(u \in (W^{1, 1}\loc \cap C^{0})(\Omega)\) is such that 
\[{}
-\Delta u + \chi_{\{\nabla u \ne 0\}} u = 0
\quad \text{in the sense of distributions in \(\Omega\),}
\]
then \(u : \Omega \to \R\) is a viscosity solution of \eqref{eqViscositySchrodinger} in the sense of \Cref{def:viscosity_sense_improved}.
\end{lemma}

\begin{proof}[Proof of \Cref{prop:viscosity}]
Since \(\Delta u \in L^{\infty}\loc(\Omega)\), by elliptic regularity theory we have \(u \in W^{2, p}\loc(\Omega)\) for every \(1 < p < \infty\) and then
\(u \in C^{1}(\Omega)\).{}
To prove that \eqref{item-2479} in \Cref{def:viscosity_sense_improved} is satisfied, take a polynomial \(P\) of degree at most \(2\) such that \(P \ge u\) in a neighborhood of \(x \in \Omega\), with \(P(x) = u(x)\) and \(\nabla P(x)\neq 0\).{}
By differentiability of \(u\), we then have 
\[{}
\nabla u(x) = \nabla P(x) \ne 0.
\]
Since \(\nabla u\) is continuous, there exists \(r > 0\) such that \(\nabla u \ne 0\) in \(B_{r}(x)\).{}
It then follows from the equation satisfied by \(u\) and elliptic regularity theory that \(u\) is smooth in \(B_{r}(x)\) and
\[{}
\Delta u = u
\quad \text{in \(B_{r}(x)\).}
\] 
By local minimality of \(P - u\) at \(x\), we then have 
\[{}
\Delta P(x) 
\ge \Delta u(x)
= u(x).
\]
Hence, \(u\) satisfies the first condition in \Cref{def:viscosity_sense_improved}.
The second one is proved in a similar way.
\end{proof}

\begin{proof}[Proof of \Cref{corollaryVorticity}]
The precise representative \(\Lebesgue h\) satisfies \(\Lebesgue{h} = h\) and \(\nabla\Lebesgue{h} = \nabla h\) almost everywhere in \(\Omega\).{}
Since \(\mu \in L^{1}\loc(\Omega)\), by \Cref{propositionVorticity} we then have
\begin{equation}\label{eqMuL1Precise}	
	\mu= \Lebesgue{h} \chi_{\{\nabla \Lebesgue{h}=0 \}}
	\quad \text{almost everywhere in \(\Omega\).}
	\end{equation}
By the equation satisfied by \(h\), we get
\[{}
-\Delta \Lebesgue h + \chi_{\{\nabla \Lebesgue{h} \ne 0\}} \Lebesgue{h} = 0
\quad \text{in the sense of distributions in \(\Omega\).}
\]	
	From elliptic regularity theory, \(\Lebesgue{h}\) thus belongs to \(C^1(\Omega)\). 

	We now decompose the relatively closed set \(G \vcentcolon= \{\nabla \Lebesgue{h} = 0\}\) as a disjoint union:
	\begin{equation}
	\label{eqZeroSetg}
	G
	= \Int{G} \cup \bigl(\partial G \cap \{\Lebesgue{h} = 0\}\bigr) \cup \bigl(\partial G \cap \{\Lebesgue{h} \neq 0\}\bigr),
	\end{equation}
	where the boundary \(\partial\) is computed with respect to the relative topology in \(\Omega\), and so yields a subset of \(\Omega\).{}
	The open set \(\Int{G}\) is a finite or countably infinite 
	union of open connected components \((U_{j})_{j \in J}\). 
	Since  \(\nabla \Lebesgue{h} = 
	0\) in each \(U_j\), then \(\Lebesgue{h}\) is a constant \(m_{j} \in \R\) in \(U_{j}\).{}
	
	By \Cref{prop:viscosity}, \(\Lebesgue{h} : \Omega \to \R\) is a viscosity solution of \eqref{eqViscositySchrodinger} in the sense of \Cref{def:viscosity_sense_improved}. 
	One then has by \cite{CaffarelliSalazar:2002}*{Lemma~9}, see also \cite{CaffarelliSalazarShahgholian:2004}, that \(\Lebesgue{h}\) is locally a \(C^{1, 1}\) function in the open set \(\{\Lebesgue{h} > 0\}\) where \(\Lebesgue{h}\) is positive.
	As the boundary  \(\partial\) is computed in the relative topology in \(\Omega\) and \(\nabla\Lebesgue{h}\) is continuous, we have \(\partial G = \partial\{\nabla \Lebesgue{h} \ne 0\}\).{}
	We can then apply \cite{CaffarelliSalazar:2002}*{Corollary~14} to deduce that the free boundary
	\[{}
	B_{+} \vcentcolon= \partial G \cap \{\Lebesgue{h} > 0\}
	\]
	has finite Hausdorff measure \(\cH^{N - 1}\).{}
	In particular, \(B_{+}\) is negligible with respect to the Lebesgue measure.
	The same argument applied to \(- \Lebesgue{h}\) in \(\{\Lebesgue{h} < 0\}\) implies that 
	\[{}
	B_{-} \vcentcolon= \partial G \cap \{\Lebesgue{h} < 0\}
	\]
	is also negligible and then so is \(B_{+} \cup B_{-}\).
	We now deduce from \eqref{eqMuL1Precise} and \eqref{eqZeroSetg} that
	\[{}
	\mu 
	=  \Lebesgue{h} \chi_{\Int{G}}
	= \sum_{j \in J}{m_{j} \chi_{U_{j}}} \quad \text{almost everywhere in \(\Omega\).}
	\qedhere
	\]
\end{proof}

\section*{Acknowledgements}
The authors would like to thank G.~Mingione, P.~Mironescu, S.~Mosconi, B.~Raita, D.~Spector and J.~Verdera for stimulating discussions.
We also thank the referee whose comments helped us improve the presentation.
L.~Ambrosio was supported by the MIUR PRIN 2015 project ``Calculus of Variations''.
A.~C. Ponce was supported by the Fonds de la Recherche scientifique (F.R.S.--FNRS) under the Crédit de recherche J.0020.18., ``Local and nonlocal problems involving Sobolev functions''.
R.~Rodiac was supported by the F.R.S.--FNRS under the Mandat d’Impulsion scientifique F.4523.17, ``Topological singularities of Sobolev maps''.

\bibliographystyle{plain}

\begin{bibdiv}
\begin{biblist}

\bib{Alberti}{article}{
      author={Alberti, Giovanni},
       title={A Lusin type property of gradients},
        date={1991},
     journal={J. Funct. Anal.},
      volume={100},
       pages={110\ndash 118},
}

\bib{AlbertiBianchiniCrippa2014a}{article}{
      author={Alberti, Giovanni},
      author={Bianchini, Stefano},
      author={Crippa, Gianluca},
       title={On the {$L^p$}-differentiability of certain classes of
  functions},
        date={2014},
     journal={Rev. Mat. Iberoam.},
      volume={30},
       pages={349\ndash 367},
}

\bib{AFP}{book}{
      author={Ambrosio, Luigi},
      author={Fusco, Nicola},
      author={Pallara, Diego},
       title={Functions of bounded variation and free discontinuity problems},
      series={Oxford Mathematical Monographs},
   publisher={Oxford University Press},
       place={Oxford},
        date={2000},
}


\bib{BenilanBrezisCrandall}{article}{
   author={B\'enilan, {Ph}ilippe},
   author={Brezis, Haim},
   author={Crandall, Michael G.},
   title={A semilinear equation in $L^{1}(\R^{N})$},
   journal={Ann. Scuola Norm. Sup. Pisa Cl. Sci. (4)},
   volume={2},
   date={1975},
   pages={523--555},
}


\bib{CaffarelliSalazar:2002}{article}{
   author={Caffarelli, Luis},
   author={Salazar, Jorge},
   title={Solutions of fully nonlinear elliptic equations with patches of
   zero gradient: existence, regularity and convexity of level curves},
   journal={Trans. Amer. Math. Soc.},
   volume={354},
   date={2002},
   pages={3095--3115},
}


\bib{CaffarelliSalazarShahgholian:2004}{article}{
   author={Caffarelli, Luis},
   author={Salazar, Jorge},
   author={Shahgholian, Henrik},
   title={Free-boundary regularity for a problem arising in
   superconductivity},
   journal={Arch. Ration. Mech. Anal.},
   volume={171},
   date={2004},
   pages={115--128},
}

\bib{CalderonZygmund}{article}{
      author={Calderón, A.~P.},
      author={Zygmund, A.},
       title={On the existence of certain singular integrals},
        date={1952},
     journal={Acta Math.},
      volume={88},
       pages={85\ndash 139},
}


\bib{CalderonZygmund1961}{article}{
      author={Calderón, A.~P.},
      author={Zygmund, A.},
   title={Local properties of solutions of elliptic partial differential
   equations},
   journal={Studia Math.},
   volume={20},
   date={1961},
   pages={171--225},
}


\bib{CalderonZygmund1962}{article}{
      author={Calderón, A.~P.},
      author={Zygmund, A.},
   title={On the differentiability of functions which are of bounded
   variation in Tonelli's sense},
   journal={Rev. Un. Mat. Argentina},
   volume={20},
   date={1962},
   pages={102--121},
}

\bib{CufiVerdera}{article}{
   author={Cuf\'\i , Juli\`a},
   author={Verdera, Joan},
   title={Differentiability properties of Riesz potentials of finite measures and non-doubling Calderón-Zygmund theory},
   journal={Ann. Scuola Norm. Sup. Pisa (5)},
   volume={18},
   date={2018},
   pages={1081--1123},
}	



\bib{DeVoreSharpley}{article}{
   author={DeVore, Ronald A.},
   author={Sharpley, Robert C.},
   title={Maximal functions measuring smoothness},
   journal={Mem. Amer. Math. Soc.},
   volume={47},
   date={1984},
   number={293},
   pages={viii+115},
}
\bib{EvansGariepy}{book}{
      author={Evans, Lawrence~C.},
      author={Gariepy, Ronald~F.},
       title={Measure theory and fine properties of functions},
      series={Textbooks in Mathematics},
   publisher={CRC Press, Boca Raton, FL},
        date={2015},
}


\bib{FrankLieb}{article}{
     author = {Frank, Rupert},
     author = {Lieb, Elliott},
      title = {A ``liquid-solid'' phase transition in a simple model for swarming, based on the ``no flat-spots'' theorem for subharmonic functions},
    journal = {Indiana Univ. Math. J.},
     volume = {67},
       year = {2018},
      pages = {1547--1569},
}	

\bib{GmeinederRaita}{article}{
	author={Gmeineder, Franz},
	author={Raita, Bogdan},
	title={On critical \(L^{p}\)-differentiability of BD-maps},
	note={To appear in Rev. Mat. Iberoam},
}


\bib{Grafakos}{book}{
   author={Grafakos, Loukas},
   title={Classical Fourier analysis},
   series={Graduate Texts in Mathematics},
   volume={249},
   edition={3},
   publisher={Springer, New York},
   date={2014},
}

\bib{Hajlasz1996}{article}{
      author={Haj{\l}asz, Piotr},
       title={On approximate differentiability of functions with bounded
  deformation},
        date={1996},
     journal={Manuscripta Math.},
      volume={91},
       pages={61\ndash 72},
}

\bib{Ishii1995}{article}{
      author={Ishii, Hitoshi},
       title={On the equivalence of two notions of weak solutions, viscosity
  solutions and distribution solutions},
        date={1995},
     journal={Funkcial. Ekvac.},
      volume={38},
       pages={101\ndash 120},
}

\bib{JuutinenLindqvistManfredi2001}{article}{
      author={Juutinen, Petri},
      author={Lindqvist, Peter},
      author={Manfredi, Juan~J.},
       title={On the equivalence of viscosity solutions and weak solutions for
  a quasi-linear equation},
        date={2001},
     journal={SIAM J. Math. Anal.},
      volume={33},
       pages={699\ndash 717},
}


\bib{KuusiMingione}{article}{
   author={Kuusi, Tuomo},
   author={Mingione, Giuseppe},
   title={Guide to nonlinear potential estimates},
   journal={Bull. Math. Sci.},
   volume={4},
   date={2014},
   pages={1--82},
}

\bib{Lions1983}{article}{
      author={Lions, P.-L.},
       title={Optimal control of diffusion processes and
  {H}amilton-{J}acobi-{B}ellman equations. {II}. {V}iscosity solutions and
  uniqueness},
        date={1983},
     journal={Comm. Partial Differential Equations},
      volume={8},
       pages={1229\ndash 1276},
}


\bib{LittmanStampacchiaWeinberger1963}{article}{
   author={Littman, W.},
   author={Stampacchia, G.},
   author={Weinberger, H. F.},
   title={Regular points for elliptic equations with discontinuous
   coefficients},
   journal={Ann. Scuola Norm. Sup. Pisa (3)},
   volume={17},
   date={1963},
   pages={43--77},
}


\bib{MaranoMosconi2015}{article}{
   author={Marano, Salvatore A.},
   author={Mosconi, Sunra J. N.},
   title={Multiple solutions to elliptic inclusions via critical point
   theory on closed convex sets},
   journal={Discrete Contin. Dyn. Syst.},
   volume={35},
   date={2015},
   pages={3087--3102},
}


\bib{Ponce2016}{book}{
      author={Ponce, Augusto~C.},
       title={Elliptic {PDE}s, measures and capacities. from the {P}oisson
  equations to nonlinear {T}homas-{F}ermi problems},
      series={EMS Tracts in Mathematics},
   publisher={European Mathematical Society (EMS), Z\"urich},
        date={2016},
      volume={23},
}

\bib{Raita2017}{article}{
author={Raita, Bogdan},
title={Critical $L^p$-differentiability of $BV^\mathbb{A}$-maps and canceling operators},
eprint={arXiv:1712.01251},
}

\bib{Rodiac2018}{article}{
      author={Rodiac, R\'emy},
       title={Description of limiting vorticities for the magnetic {2D
  Ginzburg-Landau} equations},
      note={To appear in Ann. Inst. H. Poincar\'e Anal. Non Lin\'eaire},
}

\bib{SandierSerfaty2003}{article}{
      author={Sandier, Etienne},
      author={Serfaty, Sylvia},
       title={Limiting vorticities for the {G}inzburg-{L}andau equations},
        date={2003},
     journal={Duke Math. J.},
      volume={117},
       pages={403\ndash 446},
}

\bib{SandierSerfaty}{book}{
      author={Sandier, Etienne},
      author={Serfaty, Sylvia},
       title={Vortices in the magnetic {G}inzburg-{L}andau model},
      series={Progress in Nonlinear Differential Equations and their
  Applications},
   publisher={Birkh\"auser, Boston, MA},
        date={2007},
      volume={70},
}



\bib{Stein1970}{book}{
      author={Stein, Elias~M.},
       title={Singular integrals and differentiability properties of
  functions},
      series={Princeton Mathematical Series, No.~30},
   publisher={Princeton University Press, Princeton, NJ},
        date={1970},
}


\bib{SteinShakarchi4}{book}{
      author={Stein, Elias~M.},
      author={Shakarchi, Rami},
       title={Functional analysis},
      series={Princeton Lectures in Analysis},
   publisher={Princeton University Press, Princeton, NJ},
        date={2011},
      volume={4},
}


\bib{Verdera}{article}{
author={Verdera, J.},
title={Capacitary differentiability of potentials of finite Radon measures},
eprint={arXiv:1812.11419},
}


\end{biblist}
\end{bibdiv}

\end{document}